\documentclass[a4paper,oneside,11pt]{article}

\usepackage{amsxtra,amssymb,amsmath,amsthm,amsbsy}
\usepackage{amssymb,amsmath,amsthm,mathbbol,enumerate,longtable,hyperref,cleveref}
\usepackage{graphicx}
\usepackage[all]{xy}

\usepackage{enumerate}
\usepackage{derivative}
\usepackage{tensor}
\usepackage{comment}
\usepackage{tikz}

\usepackage{xcolor}

\oddsidemargin=.\paperwidth
\topmargin=.\paperheight

\theoremstyle{plain}
\newtheorem{theorem}{Theorem}[section]
\newtheorem{lemma}[theorem]{Lemma}
\newtheorem{proposition}[theorem]{Proposition}
\newtheorem{corollary}[theorem]{Corollary}

\newtheorem{conjecture} [theorem] {Conjecture}

\theoremstyle{definition}
\newtheorem{definition}[theorem]{Definition}
\newtheorem{example}[theorem]{Example}
\newtheorem{remark}[theorem]{Remark}

\newcommand{\Man}{\operatorname{Man}}
\newcommand{\Top}{\operatorname{Top}}

\newcommand{\supp}{\operatorname{supp}}

\newcommand{\R}{{\mathbb{R}}}
\newcommand{\Z}{{\mathbb{Z}}}

\def\from{\leftarrow}

\def\xto{\xrightarrow}
\def\xfrom{\xleftarrow}
\def\toto{\rightrightarrows}

\def\id{{\rm id}}

\def\action{\curvearrowright}
\def\cosk{{\rm cosk}}
\def\sk{{\rm sk}}

\begin{document}

\title{\bf On the cohomology of differentiable stacks}
	\author{Matias del Hoyo \and Cristian Ortiz \and Fernando Studzinski}
	%
	%
	%

\date{}

	\maketitle

\begin{abstract}
Morita equivalence classes of Lie groupoids serve as models for differentiable stacks, which are higher spaces in differential geometry, generalizing manifolds and orbifolds. 
Representations up to homotopy of Lie groupoids provide a higher analog of classical representations and play a significant role in Poisson geometry.
In this paper, we prove that the cohomology with coefficients in a representation up to homotopy is a Morita invariant, and therefore an invariant of the underlying stack. This result was inspired by the 2-term case, previously developed by del Hoyo and Ortiz, and it relies on the simplicial approach to representations up to homotopy, recently introduced by del Hoyo and Trentinaglia. As a subsidiary result, we include a proof of the cohomological descent for higher Lie groupoids.
\end{abstract}

\tableofcontents


\section{Introduction} 


{\bf Lie groupoids} have shown to be useful in a variety of situations: as the global counterpart for Lie algebroids, as geometric models for noncommutative algebras, as unifying frameworks for classic geometries, and also as an intermediate step in constructing differentiable stacks. Each Lie groupoid $G$ gives rise to a {\bf differentiable stack} $[G]$, and two differentiable stacks are isomorphic, $[G]\cong[G']$, if and only if the corresponding Lie groupoids are {\bf Morita equivalent}. These equivalences can be described either by principal bi-bundles or by fractions of Morita fibrations or {\bf hypercovers} $\phi:G'\to G$. Thus, the geometry of differentiable stacks is determined by the properties and invariants of a Lie groupoid which are preserved by hypercovers. 


We can canonically identify a Lie groupoid with a simplicial manifold via its nerve. The {\bf (differentiable) cohomology} $H(G)$ is one of the main invariants of Lie groupoids, and it naturally arises from the simplicial structure. 
One can work with coefficients in a {\bf representation} $R:G \action E$ on a vector bundle over the units $E\to G_0$, and define the cohomology with {\bf coefficients} $H(G, E)$, which generalizes the previous one. Given $\phi:G'\to G$ a morphism, we can consider the {\bf pullback} representation $\phi^*R:G'\action \phi^*E$, and the induced morphism $H(G,E)\to H(G',\phi^*E)$. 
A fundamental result of the theory of Lie groupoids, originally proven by M. Crainic via bi-bundles, is the Morita invariance of cohomology:

\begin{theorem}[\cite{crainic}]\label{thm:1}
If $G',G$ are Lie groupoids, $\phi:G'\to G$ is a hypercover and $G\action E$ is a representation, then 
$\phi^*:H(G,E)\to H(G',\phi^*E)$ is an isomorphism.
\end{theorem}


Lie groupoid representations on vector bundles are too strict. A key generalization, introduced by C. Arias Abad and M. Crainic  \cite{ac}, is that of a {\bf representation up to homotopy} $R: G\action E$ on a (bounded) graded vector bundle $E=\bigoplus_{n\in\Z}E_n\to G_0$. 
In such a representation, each object is associated with a chain complex, each arrow with a chain map, and each commutative triangle with a chain homotopy. 
Then the cohomology $H(G,E)$ with coefficients in a representation up to homotopy is defined.
When $E=E_0$ this recovers the usual representations, and when $E=E_1\oplus E_0$ this is intimately related to the theory of {\bf VB-groupoids}, double structures mixing vector bundles and Lie groupoids \cite{mackenzie}. A representation up to homotopy $R:G \action (E_1\oplus E_0)$ gives rise to a VB-groupoid via the Grothendieck construction or {\bf semi-direct product} $G\ltimes_RE\to G$, and this yields an equivalence of categories \cite{gsm,dho}.


The present paper addresses the natural problem of studying the Morita invariance of cohomology with coefficients in representations up to homotopy. This has previously been achieved in the 2-term case in \cite{dho}, which is an inspiration and a motivation for the present work. Let us review the arguments there. If $V\to G$ is a VB-groupoid, the cochains of the total Lie groupoid $V$ admit two distinguished subcomplexes: the {\bf linear cochains} $C_{lin}(V)\subset C(V)$ are those which are fiberwise linear, and the {\bf projectable cochains} $C_{proj}(V)\subset C_{lin}(V)$ are those $c$ such that $c(v)=0$ and $\delta(c)(v')=0$ whenever $v,v'$ are chains of arrows in $V$ starting with a zero. The following are the key results regarding linear and projectable cochains:
\begin{enumerate}[a)]
    \item Given $V\to G$ a VB-groupoid, there is a cochain isomorphism $C(G,E^*)\cong C_{proj}(V)$, where $R:G\action E$ is the 2-term representation up to homotopy arising from $V$. \cite[Thm 5.6]{gsm};
    \item Given $V\to G$ a VB-groupoid, the inclusion $\nu:C_{proj}(V)\to C_{lin}(V)$ is a quasi-isomorphism \cite[Lemma 3.1]{cd};
    \item A linear hypercover $\phi:V'\to V$ between VB-groupoids yields an isomorphism between the linear cohomologies $\phi^*:H_{lin}(V)\cong H_{lin}(V')$ \cite[Thm 4.2]{dho}.
\end{enumerate}

There is no shift in a), unlike the original \cite{gsm}, because the dual is taken at the level of representations and not of VB-groupoids. The above results easily combine to give the following:

\begin{theorem}[\cite{dho}]
If $G',G$ are Lie groupoids, $\phi:G'\to G$ is a hypercover and $G\action E=E_1\oplus E_0$ is a representation up to homotopy concentrated in degrees 1 and 0
, then 
$\phi^*:H(G,E)\to H(G',\phi^*E)$ is an isomorphism.
\end{theorem}


This paper generalizes the previous theorem to arbitrary representations up to homotopy. An earlier version was presented in Studzinski's Thesis \cite{studzinski}, and this work expands on those findings. 
The proof follows the same lines as the 2-term case, using higher analogs for the constructions and results. The key ingredient is the semi-direct product $G\ltimes_R E$ for general representations up to homotopy $G\action E=\bigoplus_{n\geq0}E_n$, recently developed \cite{dht}, which produces a {\bf higher vector bundle} $V\to G$, a simplicial vector bundle over the nerve, generalizing VB-groupoids. We introduce {\bf projectable cochains} of a semi-direct product in Definition \ref{def:projectable}, and establish the following generalizations:
\begin{enumerate}[a')]
    \item Given $R:G\action E=\bigoplus_{n\geq0}E_n$ a  representation up to homotopy,  there is a cochain complex isomorphism $\lambda:C(G,E^*)\cong C_{proj}(G\ltimes E)$ (Proposition \ref{prop:lambda-morphism});
    \item Given $R:G\action E=\bigoplus_{n\geq0}E_n$ a representation up to homotopy, the inclusion $\nu:C_{proj}(G\ltimes_RE)\to C_{lin}(G\ltimes_RE)$ is a quasi-isomorphism (Proposition \ref{prop:proj-lin});
    \item A linear hypercover $\phi:V'\to V$ between higher vector bundles yields an isomorphism between the linear cohomologies $\phi^*:H_{lin}(V)\cong H_{lin}(V')$ (Corollary \ref{coro:cohomology-hypercover}).
\end{enumerate}


Every VB-groupoid is a semi-direct product, so our points a') and b') do generalize a) and b) above. Note that our definition of projectable cochains is only valid for semi-direct products, not general higher vector bundles. 
Point c') is a corollary of Theorem \ref{thm:higher-descent}, stating that a hypercover between higher Lie groupoids yields an isomorphism in cohomology. This fundamental result is proven on E. Getzler's slides \cite{getzler}, based on arguments of \cite{dhi}. We provide an original proof that borrows elements from there and stands for its simplicity and geometric nature, even in the Lie groupoid case. Combining a'), b') and c') we conclude:

\begin{theorem}\label{thm:main}
If $G', G$ are Lie groupoids, $\phi:G' \to G$ is a hypercover and $G\action E$ is a general representation up to homotopy, then 
$\phi^*:H(G,E)\to H(G',\phi^*E)$ is an isomorphism.
\end{theorem}

\begin{proof}
Without loss of generality, by eventually shifting the grading, we can assume that $E^*$ and $\phi^*E^*$ are on non-negative degrees, so we can build their semi-direct products, and combining a') and b'), we get $H(G,E)\cong H_{lin}(G\ltimes E^*)$
and 
$H(G',\phi^*E)\cong H_{lin}(G'\ltimes \phi^*E^*)$.
The semi-direct product commutes with pullbacks, so the following diagram is cartesian
$$\xymatrix{G'\ltimes \phi^*E^* \ar[r] \ar[d]& G\ltimes E^* \ar[d] \\ G'\ar[r]^\phi & G}$$
and $(G'\ltimes \phi^*E^*) \to (G\ltimes E^*)$ is a linear hypercover, see Lemma \ref{lemma:nice-family}.
The result follows by c').
\end{proof}


We close by commenting on two open questions. 
One could expect a generalization of Theorem \ref{thm:main} to the case when $G',G$ are higher Lie groupoids, since representations up to homotopy still make sense. Our proofs of a') and b') strongly use that the base has an inversion map, when defining $\lambda$, and when building a homotopy inverse for the inclusion $\nu$. Thus other arguments are needed. 
The second open question is the Morita invariance of the derived category $Rep^\infty[G]$, which is the localization of the category of representations up to homotopy by quasi-isomorphisms \cite{ac}. Given $G',G$ Lie groupoids and $\phi:G'\to G$ a hypercover, we believe our Theorem \ref{thm:main}
may play a role in proving that the pullback functor $\phi^*:Rep^\infty[G]\to Rep^\infty[G']$ is fully faithful, using and inner-hom construction. 

\medskip


{\bf Organization.}
Sections 2 and 3 review the background material on Lie groupoids, setting notations and conventions, and serving as a quick reference. Section 4 introduces projectable cochains and sets the isomorphism of a'). Section 5 introduces a filtration on the linear cochains and proves that each step is a quasi-isomorphism, achieving b'). Section 6, which can be read independently, reviews the fundamental notions on higher Lie groupoids, improving and simplifying some results, and showing the cohomological descent, of which c') is a corollary. 

\medskip

{\bf Acknowledgments.}
M. del Hoyo and C. Ortiz were supported by Brazil National Council for Scientific and Technological Development - CNPq, grants 310289/2020-3, 315502/2020-7 and 402320/2023-9. 
M. del Hoyo was also supported by Rio de Janeiro Research Foundation - FAPERJ, grant E-26/201.305/2021. 
F. Studzinski was partially supported by a CAPES PhD Fellowship and São Paulo Research Foundation - FAPESP, grant 2015/01698-7.


\section{Lie groupoids, hypercovers and cohomology}

Let us start with an overview of fundamental notions of Lie groupoids, hypercovers and cohomology, to set notations and conventions, and to serve the reader as a quick reference. Our references, where further details can be consulted, are \cite{crainic,mackenzie,dh,gsm,cd,bdh}.

\medskip

Given a {\bf Lie groupoid} $G$, we write $G_0$ and $G_1$ for its manifolds of objects and arrows, respectively, $s,t: G_1\to G_0$ for the source and target map, which are surjective submersions, and $m: G_2\to G_1$, $u: G_0\to G_1$ and $i: G_1 \to G_1$ for the multiplication, unit and inverse maps, which are smooth. Here $G_2\subset G_1\times G_1$ denotes the manifold of pairs of composable arrows.
A {\bf morphism} $\phi: G\to H$ is a smooth functor. We write $\phi_0,\phi_1$ for the smooth maps at the level of objects and arrows. Two morphisms $\phi,\psi: G \to H$ are {\bf isomorphic}, notation $\phi\cong\psi$, if there is a smooth natural isomorphism $\alpha: G_0\to H_1$ between them. Two Lie groupoids $G,H$ are {\bf equivalent} if there are morphisms $\phi:G\to H$, $\psi:G\to H$ such that $\psi\phi\cong\id_G$ and $\phi\psi\cong\id_H$. 



A morphism $\phi: G \to H$ is {\bf Morita} if it is smoothly fully faithful and essentially surjective, namely
(ES) the map $H_1\tensor[_s]{\times}{_{\phi_0}}G_0\rightarrow H_0$,  $(h,x)\mapsto t(h)$, is a surjective submersion, and
(FF) the following is a (good) cartesian square of manifolds $$\xymatrix{G_1 \ar[r]^{\phi_1} \ar[d]_{(t,s)} & H_1 \ar[d]^{(t,s)} \\ G_0\times G_0 \ar[r]_{\phi_0\times \phi_0} & H_0\times H_0}$$
Morita morphisms can be characterized as those inducing isomorphisms on the isotropy, the orbit space, and the normal directions \cite{dh}. While every morphism invertible up to isomorphism is Morita, not every Morita morphism admits a quasi-inverse.  
A Lie groupoid morphism $\phi:G\to H$ is a {\bf fibration} if it satisfies 
(F1) the map $\phi_0:G_0\to H_0$ is a surjective submersion, and (F2) the map $(s,\phi_1):G_1\to G_0\times_{H_0}H_1$ is a surjective submersion. 
Note that (F1) is stronger than (ES) and (F2) is weaker than (FF).
A {\bf hypercover} $\phi: G\to H$ is by definition a Morita morphism and a fibration, or in other words, a morphism satisfying (F1) and (FF).

\begin{example}
\begin{enumerate}[a)]
    \item Given $f: M\to N$ a submersion, the {\bf submersion groupoid} $M\times_NM\toto M$ has an arrow from $x$ to $y$ if and only if they belong to the same fiber. Write $\phi:(M\times_NM\toto M)\to (N\toto N)$ for the projection. Then $\phi$ is a hypercover if and only if $f$ is surjective, and $\phi$ is an equivalence if and only if $f$ admits a global section.
    \item Given $\pi:P\to M$ a principal bundle with group $K$, its {\bf gauge groupoid} $P\times^KP\toto M$ has the points of $M$ as objects and the $K$-isomorphisms as arrows.  Write $\phi:(K_x\toto x)\to(P\times^K P\toto M)$ for the inclusion of an isotropy group. Then $\phi$ is always Morita, and $\phi$ is an equivalence if and only if $P\to M$ is trivial as a bundle.
\end{enumerate}
\end{example}


Two Lie groupoids $G, H$ are {\bf Morita equivalent} if there is a third one $K$ together with hypercovers $G\from K\to H$. This is a well-defined equivalence relation, for hypercovers are closed under composition and under base-change, see e.g. \cite{dh}.  A {\bf differentiable stack} $[G]$ is the class of a Lie groupoid $G$ up to Morita equivalence. 
The geometry of differentiable stacks is determined by the properties and invariants of Lie groupoids which are preserved by hypercovers. For instance, $[G]$ is {\bf separated} if $G_1\to G_0\times G_0$ is proper and
the {\bf dimension} of $[G]$ is
$2\dim G_0-\dim G_1$, are examples of Morita invariants. 


\begin{example}
\begin{enumerate}[a)]
\item We can see a manifold $M$ as a differentiable stack via the unit Lie groupoid $M\toto M$. 
It is separated and $\dim [M]=\dim M$. 
A Lie groupoid is Morita equivalent to a manifold if and only if it is a submersion groupoid.
\item A Lie group $K\toto \ast$ yields the so-called {\bf classifying stack} $[K]$, for it classifies principal bundles via stacky maps. This is a finite-dimensional model for Milnor's classifying space $BK$. It is separated if and only if $K$ is compact, and $\dim[K]=-\dim K$. A Lie groupoid is Morita equivalent to a Lie group if and only if it is a gauge groupoid.
\end{enumerate}
\end{example}



The cohomology of a Lie groupoid is defined via its nerve, which is a simplicial manifold. Let us quickly review these notions. Write $\Delta$ for the category of finite ordinals $[n]=\{n,\dots,1,0\}$, $\delta_i:[n-1]\to[n]$ for the injection missing $i$, and $\sigma_j:[n+1]\to[n]$ for the surjection repeating $j$.
 A {\bf simplicial manifold} is a contravariant functor $X:\Delta^\circ\to\Man$, or equivalently, a system $(X_n,d_i,s_j)$ where $X_n$ are manifolds, $n\geq0$, and $d_i=\delta_i^*:X_n\to X_{n-1}$ and $s_j=\sigma_j^*:X_n\to X_{n+1}$, $0\leq i,j\leq n$, are {\bf faces} and {\bf degeneracies} subject to the usual simplicial identities \cite{gj}. 
An $n$-simplex $x\in X_n$ can be regarded as a map $x:\Delta^n\to X$, where $\Delta^n=\Delta(-,[n])$ is the representable {\bf $n$-simplex}. 
Write $\iota_p:[p]\to[n]$ and $\tau_q:[q]\to[n]$ for the first and last inclusions. 
Thus $\tau_q^*(x)=x\tau_q\in X_q$ is the $q$-th back face of $x$, and $\iota_p^*(x)=x\iota_p\in X_p$ is the $p$-th front face. The {\bf cochain algebra} $C(X)=\bigoplus_{n\geq0} C^\infty(X_n)$ is the differential graded algebra given by
\begin{enumerate}[$\bullet$]
    \item $\cup:C^\infty(X_q)\times C^\infty(X_p)\to C^\infty(X_{p+q})$, $(f\cup g)(x)=f(x\tau_q)g(x\iota_p)$;
    \item $\delta:C^\infty(X_n)\to C^\infty(X_{n+1})$, $\delta(f)(x)=(-1)^{n+1}\sum_{i=0}^{n+1}(-1)^{i} f(x\delta_i)$.
\end{enumerate}
The {\bf cohomology} $H(X)$ is that of its cochain algebra. A cochain $f\in C_N(X)$ is {\bf normalized} if it vanishes over degenerate simplices. The normalized cochains define a differential graded subalgebra, and by the Dold-Kan Theorem, the inclusion $C_N(X)\to C(X)$ is a quasi-isomorphism, so $H(X)$ can be computed using just normalized cochains, see \cite{gj}.

Given $G$ a Lie groupoid, its {\bf nerve} $NG=(G_m,d_i,s_j)$, is the simplicial manifold whose $m$-simplices are maps $g:[n]\to G$, where we regard $[n]$ as a (discrete) category with an arrow $j\from i$ if and only if $j\geq i$, or equivalently, sequences of $n$ composable arrows:
$$g=(x_n\xfrom{g_n} x_{n-1}\from \cdots \from x_2\xfrom{g_2}x_1\xfrom{g_1}x_0).$$
Then the face $d_i:G_n\to G_{n-1}$ deletes $x_i$, erasing the first or last arrow or composing the two adjacent ones, and the degeneracy $s_j: G_n\to G_{n+1}$ repeats $x_j$, inserting an identity.
Moreover, $g\tau_q\in G_q$ is formed by the last $q$ arrows and $g\iota_p\in G_p$ is formed by the first $p$ arrows:
$$g\tau_q=(x_n\xfrom{g_n} x_{n-1}\xfrom{g_{n-1}} \cdots \xfrom{g_{n-q+1}} x_{n-q})
\qquad
g\iota_p=(x_p\xfrom{g_p} x_{p-1}\xfrom{g_{p-1}} \cdots \xfrom{g_{1}} x_{0}).$$
The cochain algebra and the cohomology of a Lie groupoid $C(G)$ and $H(G)$ are those of $NG$. Since the nerve is fully faithful, by an abuse of notation, we often write $G$ instead of $NG$.

\begin{example}
\begin{enumerate}[a)]
    \item If $M$ is a manifold, viewed as a unit groupoid, then $C_N(M)=H(M)=C^\infty(M)$ is the algebra of smooth functions concentrated in degree 0.
    \item If $K$ is a Lie group, viewed as a groupoid with a single object, then $H^0(K)=0$, $H^1(K)=\hom(K,\R)$ and $H^2(K)$ controls the central extensions $0\to\R\to\tilde K\to K\to 0$. 
\end{enumerate}    
\end{example}


Let $G$ be a Lie groupoid and $E\to G_0$ be a (real, finite-dimensional) vector bundle. A {\bf representation} $R:G\action E$ is a smooth map $R:G_1 \times_\pi E\to E$, $R(g,x)=R_g(x)$, such that for every $y\xfrom g x$ the map $R_g:E_x\to E_y$ is linear, $R_h R_g=R_{hg}$, and  $R_{u(x)}=\id_{E_x}$.
Given $R:G\action E$, the {\bf $E$-valued cochain complex} 
$C(G,E)=\bigoplus_n \Gamma(G_n,t^*E)$ is the differential graded $C(G)$-module, where $t=\tau_0^*:G_n\to G_0$ is the generalized target map, and 
\begin{enumerate}[$\bullet$]
    \item $\cdot:C^q(G,E)\times C^p(G)\to C^{p+q}(G,E)$, $(c\cdot f)(g)=c(g\tau_q)f(g\iota_p)$.
    \item $\delta_R:C^n(G,E)\to C^{n+1}(G,E)$, $\delta(c)(g)=R_{g_{n+1}}(c(d_{n+1}g))+ (-1)^{n+1}\sum_{i=0}^{n}(-1)^{i}c(d_{i}g)$.
\end{enumerate}
The {\bf cohomology with coefficients} $H(G,E)$ is that of $C(G,E)$. When $E=\R_{G_0}$ and $R_g=\id_{\R}$ for every $g$ this recovers $H(G)$.
Note that the graded $C(G)$-module $C(G,E)$ does not depend on the representation $R$, and that $R$ is completely determined by the differential $\delta_R$, see \cite{gsm}.

\begin{example}
\begin{enumerate}[a)]
    \item A representation of a submersion groupoid $(M\times_NM\toto M)\action (E\to M)$ is the same as a descent data for $E$. In particular, if $U\to M$ is an open cover, 
    a representation $(U\times_MU\toto U)\action \R^k_U$ is the same as a cocycle defining a vector bundle over $M$.
    \item Representations of a Lie group $K\toto\ast$ are the usual ones, and a representation of a gauge groupoid $P\times^KP\toto M$ on a vector bundle $E\to M$ corresponds to a reduction of the structure group. 
\end{enumerate}   
\end{example}

Given $\phi:G'\to G$ a Lie groupoid map and $R:G\action E$ a representation, the {\bf pullback representation} $\phi^*R:G\action\phi^*E$ is given by $(\phi^*R)(g',e)=R(\phi_1(g),e)$. There is a canonical map $\phi^*:H(G,E)\to H(G',\phi^*E)$, and if $\phi$ is a hypercover, then by Theorem \ref{thm:1} this is an isomorphism. This can be seen as a first instance of the Morita invariance of cohomology.


A {\bf VB-groupoid} $V\to G$ is a groupoid morphism such that $V_i\to G_i$ are vector bundles, and the structure maps of $V$ are linear over those of $G$. 
A VB-groupoid is a fibration. We define its {\bf core} as $E=E_1\oplus E_0$, where $E_0=V_0$ and $E_1=(\ker s:V_1\to V_0)|_{G_0}$. Its {\bf anchor map} is $t:E_1\to E_0$. 

\begin{example}
 \begin{enumerate}[a)]
     \item The tangent $TG_1\toto TG_0$ and the cotangent $T^*G_1\toto A_G^*$ of a Lie groupoid $G$ are the paradigmatic examples of VB-groupoids.
     \item Given $G\action E$ a representation, its {\bf semi-direct product} $G\ltimes E$ is a new Lie groupoid, with $(G\ltimes E)_0=E$, $(G\ltimes E)_1=G_1\times_{G_0}E$, source the projection, target the action, and multiplication, units and inverses coming from those of $G$. The projection $(G\ltimes E)\to G$ is a VB-groupoid. Moreover, a VB-groupoid is a semi-direct product if and only if $E_1=0$.
\end{enumerate}   
\end{example}

%

Given $V\to G$ a VB-groupoid, the {\bf linear cochains} $C_{lin}(V)\subset C(V)$ are the fiber-wise linear, so they identify with sections of the dual bundle, $C_{lin}^k(V)=\Gamma(V_k^*)$.
This inclusion admits a natural retraction, the {\bf linearization map} $L_V:C(V)\to C_{lin}(V)$, $L(f)(v)=\odv{}{t}_{t=0}{f(tv)}$, which commutes with the differential. Writing $K_V$ for the kernel of $L_V$, we conclude:

\begin{lemma}[\cite{cd}, Prop 2.3]\label{lemma:splitting-vb}
Given $V\to G$ a VB-groupoid, the linearization map yields natural direct sum decompositions:
$$C(V)=C_{lin}(V)\oplus K_V \qquad H(V)=H_{lin}(V)\oplus H(K_V).$$
\end{lemma}


Given $V\to G$ a VB-groupoid, there is yet another subcomplex of the cochain algebra of $V$ that we will work with. Let $V^0_n=\ker ((\iota_1)^*:V_n\to V_1)\subset V_n$ be the chains of composable arrows in $V$ starting with a 0. Then a cochain $c\in C^n_{lin}(V)$ is {\bf projectable} if $c$ and $\delta(c)$ vanish over $V^0_n$ and $V^0_{n+1}$, respectively. In other words, it must satisfy $c(v_p,\dots,v_2,0_{g_1})=0$ and $\delta(c)(v_{p+1},\dots,v_2,0_{g_1})=0$. 
Assuming that $c$ vanishes over $V^0_n$, it is easy to see that $\delta(c)$ vanishes over $V^0_{n+1}$ if and only if $c$ satisfies 
$c(v_{p},\dots,v_2,v_1)=c(v_{p},\dots,v_2,v_10_{g})$, where $0_g\in G_1\subset V_1$ is an arrow over the zero section. This equivariant condition will be crucial in our generalization for higher vector bundles. The complex of projectable cochains $C_{proj}(V)$ was introduced in \cite{gsm} with the notation $C_{VB}(V^*)$, while \cite{cd} and \cite{dho} denote it by $C_{VB}(V)$. 
It provides a smaller model for the linear cohomology, namely:

\begin{lemma}[\cite{cd}, Lemma 3.1]\label{lemma:proj-lin}
Given $V\to G$ a VB-groupoid, the inclusion $C_{proj}(V)\subset C_{lin}(V)$ is a quasi-isomorphism.
\end{lemma}

\section{Representations up to homotopy}

Representations up to homotopy are representations of Lie groupoids on graded vector bundles. They were introduced in \cite{ac} to make sense of the adjoint and coadjoint representation. The 2-term case was carefully studied in \cite{gsm}. We review these notions here, as well as the simplicial approach via the semi-direct product recently developed in \cite{dht}. We work only with graded vector bundles which are bounded from above and below. 

\medskip


Given $G$ a Lie groupoid, write $s=\iota_0^*:G_n\to G_0$ and $t=\tau_0^*G_n\to G_0$ for the generalized source and target maps. Given $E= \bigoplus_{n\in\Z} E_n$ a graded vector bundle over $G_0$, a {\bf representation up to homotopy}
$R=(R_m)_{m\geq0}:G\action E$ is given by sections 
$R_m\in\Gamma(G_m,\hom^{m-1}(s^*E,t^*E))$, so for each $g\in G_m$ and $n\geq 0$ we have 
$(R_m^g)_n:E^{s(g)}_n\to E^{t(g)}_{n+m-1}$,
such that:
\begin{enumerate}[(RH1)]
    \item $R_1(u(x))=\id_{E^x}$ and $R^g_m=0$ whenever $m\geq 2$ and $g$ is degenerate.
    \item $\sum_{k=1}^{m-1}(-1)^kR_{m-1}^{g\delta_k}=\sum_{k=0}^m(-1)^kR_{m-k}^{g\tau_{m-k}}\circ 
R_k^{g\iota_k}$ for every $g\in G_m$, $m\geq 0$.
\end{enumerate}


Let us briefly describe the sections $R_m$. When $m=0$, for each $x_0$ in $G_0$, we have a linear maps $R^{x_0}_0:E^{x_0}_n\to E^{x_0}_{n-1}$ turning $(E_{\bullet},R_0)$ into a chain complex of vector bundles over $G_0$; when $m=1$, for each arrow $x_1\xfrom {g_1} x_0$ arrow in $G_1$ we have
linear maps  $R_1^{g_1}:E_n^{x_0}\to E_n^{x_1}$ defining a chain maps between the fibers; when $m=2$, for each pair of composable arrows $(x_2 \xfrom{g_2} x_1\xfrom{g_1} x_0)$ in $G_2$, we have linear maps $R_2^{g_2,g_1}:E_n^{x_0}\to E_{n+1}^{x_2}$ giving a chain homotopy $R_1^{g_2}\circ R_1^{g_1}\cong R_1^{g_2g_1}$.
 When $m>2$, $R^g_m$ can be thought of as a higher homotopy giving constraints for the lower homotopies. 

\begin{example}
\begin{enumerate}[a)]
    \item When $E=E_0$ this recovers the previous notion of a representation of a Lie groupoid on a vector bundle;
    \item When $M$ is a manifold, viewed as a unit groupoid, then a representation up to homotopy $M\action E$ is the same as a differential $\partial:E_n\to E_{n-1}$.
    \item Given a Lie groupoid $G$, its tangent VB-groupoid $TG$ can be split via the choice of a connection or cleavage $\Sigma$, giving rise to the {\bf adjoint representation} $Ad_\Sigma:G\action (A\oplus TM)$ \cite{ac}, this is a 2-term representation and one of the motivating examples. 
    \item
    Generalizing the previous example, given a VB-groupoid $V\to G$, it is possible to define a 2-term representation up to homotopy $R_\Sigma: G\action E=E_1\oplus E_0$, by choosing a {\bf cleavage} $\Sigma$, namely a section of the source $s: V_1\to s^*V_0$ preserving the units. See \cite{gsm,dho}.
\end{enumerate}
\end{example}


There is an alternative equivalent approach to representations up to homotopy using differential graded algebras \cite{ac}. We review it here to set notations and conventions, some of ours are different than those there. 
Given $G$ a Lie groupoid and $E = \bigoplus_{n\in\Z} E_n$ a (bounded) graded vector bundle over $G_0$, an {\bf $E$-valued cochain} of bi-degree $(i,j)$ is an element in 
$$c\in C^{i,j}(G, E) =C^i(G, E_{-j})=\Gamma(G_i,t^*E_{-j})$$
A cochain is {\bf normalized} if it vanishes over the degenerate simplices.
We write $C^n(G,E)=\bigoplus_{i+j=n}C^{i,j}(G, E)$, and $C(G, E)=\bigoplus_{n\in\Z}C^n(G, E)$. 
This is a $C^{\bullet}(G)$-module, with 
$$(c\cdot f) (g) = c(g\tau_q) f(g\iota_p) \qquad g\in G_{p+q},\ c\in C^q(G, E),\ f\in C^p(G).$$

\begin{lemma}[\cite{ac}]\label{lemma:R-d}
A representation up to homotopy $R:G\action E$ is the same as a degree 1 differential $d$ on $C(G,E)$ preserving the normalized cochains. The {\bf cohomology with coefficients} $H(G,E)$ is that of the complex $(C(G,E),d)$.
\end{lemma}

\begin{proof}
A map $d_m:C^{i,j}(G,E)\to C^{i+m,j-m+1}(G,E)$ define and is defined by a map $R_m:s^*E_j\to t^*E_{j+m-1}$ via the following formulas:
\begin{itemize}
\item $d_m(c)(g) = (-1)^{j}R_m^{g\tau_m}c(g\iota_i)$, where $g\in G_{m+i}$, $m\neq 1$
\item $d_1(c)(g) = (-1)^{j} R_1^{g\tau_1}c(d_{i+1}(g)) + (-1)^{i+j+1}\sum_{r=0}^i (-1)^{r} c(d_{r}(g))$, where $g\in G_{1+i}$
\end{itemize}
With these correspondences, setting $d=\sum_{m\geq0} d_m$, we can directly check that (RH1) holds if and only if $d$ preserves the normalized cochains, and that (RH2) holds if and only if $d^2=0$.
\end{proof}







Let $E=\bigoplus_{n\in\Z} E_n$ be a graded vector bundle over $G_0$, and $R: G\action E$ be a representation up to homotopy. We write $E^*= \bigoplus_{n\in\Z} E^*_{-j}$ for the dual vector bundle, so $(E^*)_j=(E_{-j})^*$. Given $g=(g_n,\dots,g_1)\in G_n$, write $g^{-1}=(g_1^{-1},\dots,g_n^{-1})\in G_n$ for the inverse chain of composable arrows. The {\bf dual representation} $R^*:G\action E^*$ is the one given by the following formula: 
$$(R^*)_{m}^g(\varphi)(v)=\varphi(R^{g^{-1}}_m(v)) \qquad g\in G_m,\ \varphi\in (E^{x_0}_{-j})^*=(E^*)_j^{x_0},\ v\in E^{x_m}_{-j-m+1}$$
We have
$C^{i,j}(G,E^*) = C^{i}(G, (E^*)_{-j})  =  
\Gamma(G_i, t^*E^*_{j})$
and $C^n(G,E^*)= \bigoplus_{i+j=n}\Gamma(G_i,t^*E^*_{j})$.
Our sign convention differs from that in \cite{ac}. 

\begin{lemma}\label{lemma:d-dual}
If $c\in C^{i,j}(G,E^*)=\Gamma(G_i,t^*E^*_{j})$ is an $E^*$-valued cochain of bidegree $(i,j)$, then  $d_m(c)\in C^{i+m,j-m+1}(G,E^*)=\Gamma(G_{i+m},t^*E^*_{j-m+1})$ is given by
\begin{itemize}
    \item  $d_m(c)(g)(v) =  (-1)^{j}c(g\iota_i) (R_m^{(g\tau_m)^{-1}}(v))$, where $g\in G_{m+i}$, $m\neq 1$
    \item $d_1(c)(g)(v) =  (-1)^{j}c(g\iota_i) (R_1^{(g\tau_1)^{-1}}(v)) + (-1)^{i+j+1}\sum_{r=0}^{i}(-1)^r c(d_rg)(v)$, where $g\in G_{1+i}$.
\end{itemize}
\end{lemma}

\begin{proof}
It follows by combining the formulas in Lemma \ref{lemma:R-d} and the definition of $R^*_m$.
\end{proof}


We mentioned that any VB-groupoid can be split into a 2-term representation up to homotopy via a cleavage.
On the other direction, if $R:G\action E_1\oplus E_0$ is a 2-term representation up to homotopy, then its {\bf semi-direct product} or Grothendieck construction is a VB-groupoid $(G\ltimes_RE)\to G$ with a canonical cleavage \cite{gsm}. This construction sets an equivalence of categories between VB-groupoids and 2-term representations up to homotopy \cite[Thm 2.7]{dho}. 
The next result relates the semi-direct product with the projectable cochains.

\begin{lemma}[\cite{gsm}, Thm 5.6]
Given $G\action (E_1\oplus E_0)$ a 2-term representation up to homotopy, there is a canonical isomorphism $\lambda:C(G,E^*)\cong C_{proj}(G\ltimes E)$.
\end{lemma}

Our formulation does not have a shift, contrary to the original one in \cite{gsm}, because we take the dual representations instead of the dual VB-groupoid. 
In the next section, we will extend the above result to more general representations up to homotopy. We will use the generalized semi-direct product, recently introduced in \cite{dht}, that we now recall. 
Write $\chi_i:[0]\to[n]$ for the $i$-th inclusion, and $x_i=(\chi_i)^*:G_n\to G_0$ for the induced {\bf vertex map}. 
Given $\beta:[l]\to[m]$, write $\beta':[l+1]\to[m+1]$ for the map given by $\beta'(0)=0$ and $\beta'(i+1)=\beta(i)+1$, $i>0$.

\begin{definition}[\cite{dht}]\label{def:sdp}
Given $E=\bigoplus_{n\geq0} E_n$ a (bounded) graded vector bundle over $G_0$, and given $R:G\action E$ a representation up to homotopy, the {\bf semi-direct product} $(G\ltimes_R E,d_i,s_j)$ is the simplicial vector bundle over the nerve $NG$ defined by
$$(G\ltimes_R E)_n=\bigoplus_{\substack{[k]\xto\alpha[n] \\ \alpha\text{ injective}\\\alpha(0)=0}} x_{\alpha(k)}^*E_k$$
Write $\pi_\beta:(G\ltimes_RE)_n\to E_l$ for the projection over $x_{\beta(l)}:G_n\to G_0$. The faces $d_i:(G\ltimes_R E)_n\to(G\ltimes_R E)_{n-1}$, $i>0$, and the degeneracies $s_j:(G\ltimes_R E)_n\to(G\ltimes_R E)_{n+1}$ are given by
$$\pi_\beta s_j=\pi_{\sigma_j\beta}\qquad\pi_\beta d_i=\pi_{\delta_i\beta}$$
The face $d_0:(G\ltimes_R E)_n\to(G\ltimes_R E)_{n-1}$ encodes the information of $R$ and is given by
$$\pi_\beta d_0=
\sum_{k=0}^{l+1} (-1)^{l}R_{l+1-k}^{g\beta'\tau_{l+1-k}}\pi_{\beta'\iota_k}
-\sum_{i=1}^{l}(-1)^i \pi_{\beta'\delta_i}$$ 
\end{definition}


The above formulas give a common generalization to the Dold-Kan construction for chain complexes of vector bundles over a manifold, and to the Grothendieck construction or semi-direct product of 2-term representations up to homotopy discussed in \cite{gsm,dho}. See \cite[\S 6]{dht} for further details.  


\begin{remark}\label{rmk:homogeneous-vectors}
Given $g\in G_n$, $\alpha:[k]\to[n]$ and $e\in E_k^{x_{\alpha(k)}(g)}$, we write $(e,\alpha,g)\in (G\ltimes_RE)_n$ for the corresponding homogeneous vector. Then we can rewrite the faces as follows:
$$d_i(e,\alpha,g)=
\begin{cases}
(e,\sigma_{i-1}\alpha,g\delta_i) & i\notin\alpha \\
\sum_{\beta/\beta'\iota_k=\alpha}(-1)^{l}(R_{l+1-k}^{g\beta'\tau_{l+1-k}}(e),\beta,g\delta_0)-\sum_{\beta/\beta'\delta_r=\alpha}(-1)^r(e,\beta,g\delta_0) & i=0\\
 0 & i\in\alpha,\ i>0
\end{cases}$$
When $i=0$, the first sum is over all injective maps $\beta:[l]\to[n-1]$ such that $\beta'\iota_k=\alpha$, and the second sum is over all $\beta$ such that $\beta'\delta_r=\alpha$.
We can visualize each $\alpha$ as a subset of $[n]$, and study how the support changes when applying $d_0$. If $\pi_\beta d_0(e,\alpha,g)\neq0$, then $\beta'$ is  obtained from $\alpha$ by either adding some elements above $\alpha(k)$, or by adding one element before $\alpha(k)$.
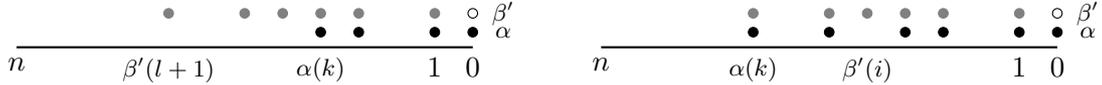
$$\begin{tikzpicture}
\draw[thick](6,0)node[below] {$0$} --(0,0) node[below] {$n$};
\draw[fill] (6,0.2) circle (1.8pt);
\draw (6.4,0.2) node {\footnotesize $\alpha$};
\draw[fill] (5.5,0.2) circle (1.8pt);
\draw[fill] (4.5,0.2) circle (1.8pt);
\draw[fill] (4,0.2) circle (1.8pt);
\draw (5.5,0) node[below] {$1$};
\draw (6,0.45) circle (1.8pt);
\draw (6.4,0.45) node {\footnotesize $\beta'$};
\draw[fill,gray] (5.5,0.45) circle (1.8pt);
\draw[fill,gray] (4.5,0.45) circle (1.8pt);
\draw[fill,gray] (4,0.45) circle (1.8pt);
\draw[fill,gray] (3.5,0.45) circle (1.8pt);
\draw (4,0) node[below] {\footnotesize $\alpha(k)$};
\draw[fill,gray] (3,0.45) circle (1.8pt);
\draw[fill,gray] (2,0.45) circle (1.8pt); \draw (2,0) node[below] {\footnotesize $\beta'(l+1)$};
\end{tikzpicture}
\qquad
\begin{tikzpicture}
\draw[thick](6,0)node[below] {$0$} --(0,0) node[below] {$n$};
\draw[fill] (6,0.2) circle (1.8pt);
\draw (6.4,0.2) node {\footnotesize $\alpha$};
\draw[fill] (5.5,0.2) circle (1.8pt);
\draw[fill] (4.5,0.2) circle (1.8pt);
\draw[fill] (4,0.2) circle (1.8pt);
\draw[fill] (3,0.2) circle (1.8pt);
\draw[fill] (2,0.2) circle (1.8pt); 
\draw (6,0.45) circle (1.8pt);
\draw (6.4,0.45) node {\footnotesize $\beta'$};
\draw[fill,gray] (5.5,0.45) circle (1.8pt);
\draw[fill,gray] (4.5,0.45) circle (1.8pt);
\draw[fill,gray] (4,0.45) circle (1.8pt);
\draw[fill,gray] (3.5,0.45) circle (1.8pt);
\draw[fill,gray] (3,0.45) circle (1.8pt);
\draw[fill,gray] (2,0.45) circle (1.8pt);
\draw (5.5,0) node[below] {$1$};
\draw (3.5,0) node[below] {\footnotesize $\beta'(i)$};
\draw (2,0) node[below] {\footnotesize $\alpha(k)$};
\end{tikzpicture}$$
In particular, if $\alpha(1)>1$, there are no $\beta$ such that $\beta'\iota_k=\alpha$, and $d_0(v,\alpha,g)=(v,\sigma_0\alpha,g\delta_0)$.
\end{remark}




Given $R: G\action E$ a representation up to homotopy, the semi-direct product $(G \ltimes_R E) \to G$ is a {\bf higher vector bundle}, namely a simplicial vector bundle that is also a simplicial fibration. Unlike the VB-groupoid case, higher vector bundles may not be isomorphic to semi-direct products, see \cite{dht}.


\section{Linear and projectable cochains}



In this section, we introduce linear and projectable cochains for higher vector bundles arising from a representation up to homotopy. The main result here is Proposition \ref{prop:lambda-morphism}, generalizing \cite[Thm 5.6]{gsm}, and giving the isomorphism a') mentioned in the introduction.

\medskip



\medskip




Let $G$ be a Lie groupoid, and let $V\to G$ be a higher vector bundle, namely a simplicial vector bundle over the nerve that is also a simplicial fibration. As in the VB-groupoid case, a cochain $f\in C^p(V)$ is {\bf linear} if it is fiber-wise linear. This determines a subcomplex $C_{lin}(V)\subset C(V)$.
The {\bf linearization map} $L_V:C(V)\to C_{lin}(V)$, $L(f)(v)=\odv{}{t}_{t=0}{f(tv)}$, with kernel $K_V$, is a retraction as in the VB-groupoid case, and we get the straightforward generalization of Lemma \ref{lemma:splitting-vb}:

\begin{lemma}\label{lemma:splitting}
Given $V\to G$ a higher vector bundle, the linearization map yields  natural direct sum decompositions:
$$C(V)=C_{lin}(V)\oplus K_V \qquad H(V)=H_{lin}(V)\oplus H(K_V).$$
\end{lemma}


We will not define projectable cochains for general higher vector bundles, but only those arising from a semi-direct product, as in Definition \ref{def:sdp}. Let $R:G\action E$ be a representation up to homotopy on $E=\oplus_{j=0}^N E_j$, and let $G\ltimes_R E$ be the semi-direct product. The linear cochains
$(C_{lin}(G\ltimes_R E),\delta)$ identify with sections of the dual vector bundles
$$C^{n}_{lin}(G\ltimes_R E)=
\Gamma(G_n,\bigg(\bigoplus_{\substack{[k]\xto\alpha[n] \\ \alpha(0)=0}} x_{\alpha(k)}^*E_k\bigg)^*)=
\bigoplus_{\substack{[k]\xto\alpha[n] \\ \alpha(0)=0}}\Gamma(G_n, x_{\alpha(k)}^*E_k^*).$$
Given $\beta:[j]\to[i+j]$ and $c\in\Gamma(G_{i+j}, x_{\beta(j)}^*E_j^*)$, we write $(c,\beta)\in C^{i+j}_{lin}(G\ltimes E)$ for the corresponding $\beta$-homogeneous linear cochain, and we say that $(c,\beta)$ has {\bf bi-degree} $(i,j)$. Thus,
$$C^{i,j}_{lin}(G\ltimes_RE)= \bigoplus_{\substack{[j]\xto\alpha[i+j] \\ \alpha(0)=0}}\Gamma(G_{i+j}, x_{\alpha(j)}^*E_j^*)$$
This natural bigrading allows us to decompose the differential as $\delta=\sum_{m\geq 0} \delta_{m}$, with 
$\delta_m:C^{i,j}_{lin}(G\ltimes_R E)\to C^{i+m,j-m+1}_{lin}(G\ltimes_R E)$.

\begin{lemma}\label{lemma:d-linear}
With the above notations, and writing $(g,v)=\sum_\alpha (g,\alpha,v_\alpha)\in (G\ltimes_RE)_{i+j+1}$,  the differential of a linear cochain is given by:
\begin{itemize}
    \item $\delta_m(c,\beta)(g,v)=(-1)^j c(d_0g)R_{m}^{g\beta'\tau_m}(v_{\beta'\iota_{j-m+1}})$ whenever $m\neq 1$
    \item $\delta_{1}(c,\beta)(g,v)=\sum_{r=1}^{i+j+1}(-1)^r c(d_rg)(v_{\delta_r\beta})+\sum_{r=1}^j (-1)^{r-1}c(d_0g)(v_{\beta'\delta_r})+$\\
    $\phantom{x}\hspace{250pt}+(-1)^j c(d_0g) R_{1}^{g\beta'\tau_1}(v_{\beta'\iota_{j}})$
\end{itemize}
\end{lemma}

\begin{proof}
By definition we have $\delta(c)=\sum_{r=0}^{n+1}(-1)^r c\circ d_n$. Using the explicit equations of the face maps of the semi-direct product we get
\begin{align*}
\delta(c,\beta)(g,v)=
\sum_{r=1}^{n+1}(-1)^r c(d_rg)(v_{\delta_r\beta})
&+\sum_{r=1}^l (-1)^{r-1}c(d_0g)(v_{\beta'\delta_r})+\\
&+(-1)^l\sum_{k=0}^{l+1}c(d_0g)R_{l+1-k}^{g\beta'\tau_{l+1-k}}(v_{\beta'\iota_{k}})
\end{align*}
The result follows by grouping the above terms with respect to the bi-grading.
\end{proof}


We introduce next a generalization of projectable cochains for higher vector bundles $V\to G$ which are semi-direct products. 
Given $c=(c_\alpha)_\alpha\in C^{n}_{lin}(G\ltimes_R E)=\bigoplus_{\alpha}\Gamma(G_n, x_{\alpha(k)}^*E_k^*)$, its {\bf support} is $\supp c=\{\alpha:c_\alpha\neq0\}$. we say that an index $\alpha$ is {\bf regular} if $\alpha=\iota_k$ for some $k$.

\begin{definition}\label{def:projectable}
A cochain $c=(c_\alpha)_\alpha\in C^{n}_{lin}(G\ltimes_R E)$ is {\bf projectable} if both $\alpha$ and $\delta(\alpha)$ are supported over regular indices.
\end{definition}

\begin{remark}
If $V\to G$ is a VB-groupoid with core $E_1\oplus E_0$, we can always pick a normal cleavage $\Sigma$, yielding an isomorphism $V\cong G\ltimes_R E$. Then $c\in C^n_{lin}(V)=C^n_{lin}(G\ltimes_RE)=
\Gamma(G_n, x_{0}^*E_0^*) \oplus\bigoplus_{j>0}\Gamma(G_n, x_{j}^*E_1^*)$ is such that $\supp c\subset\{\iota_0,\iota_1\}$ if and only if $c(v)$ only depends on the first arrow of $v$. This shows that our definition is a good generalization of the projectable cochains of VB-groupoids, see \cite[Section 4]{dho}, and the original \cite[Definition 5.1]{gsm}.
\end{remark}


It is clear from the definition that projectable cochains define a subcomplex $C_{proj}(G\ltimes_R E)\subset C_{lin}(G\ltimes_R E)$. The next lemma gives a useful alternative characterization of projectable cochains via an equivariant condition, analogously to the VB-groupoid case. 
Given $g\in G_n$, $(d_0s_k)^kg=g(\sigma_k\delta_0)^k$ is the chain obtained from $g$ by replacing the first $k$ arrows by identities:
$$(d_0s_k)^kg=(x_n\xfrom{g_n} \cdots \xfrom{g_{k+2}}x_{k+1}\xfrom{g_
{k+1}}x_k\xfrom{\id}x_k\xfrom\id\dots\xfrom\id x_k).$$

\begin{lemma}\label{lemma:characterization-projectable}
If $c=(c_\alpha)$ is supported over regular indices, then $\delta(c)$ also is if and only if
 $c_{\iota_k}\in\Gamma(G_n,x^*_kE_k^*)$ does not depend on the first $k$ arrows, namely $c_{\iota_k}(g)=c_{\iota_k}((d_0s_k)^kg)$ for all $0\leq k\leq n$ and $g\in G_n$.
\end{lemma}

\begin{proof}
First, consider a homogeneous cochain $c=(c_k,\iota_k)$. By Lemma \ref{lemma:d-linear}, we have that $\delta(c_k,\iota_k)=c_k'+c''_k$, with $c'_k$ supported on the regular indices, and $c''_k$ given by  
$$c''_k(g,v)=\sum_{r=1}^{k}(-1)^r c_k(d_rg)(v_{\delta_r\iota_k})
+\sum_{r=1}^k (-1)^{r-1}c_k(d_0g)(v_{\iota_{k+1}\delta_r})$$
Then $c$ is projectable if and only if $c''_k=0$, and since $\delta_r\iota_k=\iota_{k+1}\delta_r$, $1\leq r\leq j$, this is the case if and only if 
$c(d_rg)=c(d_0g)$ for every $r$, $1\leq r\leq k$, and every $g\in G_{n+1}$. This is clearly the case if $c$ does not depend on the first $k$ arrows. Conversely, if $c(d_rg)=c(d_0g)$ for every $r$, $1\leq r\leq k$, and every $g\in G_{n+1}$, given $g\in G_n$, we have 
$c(g)=c(d_ks_kg)=c(d_0s_kg)$, and therefore, $c(g)=c((d_0s_k)^kg)$.

Consider now a general cochain $c$ supported on regular indices, so $c=\sum_k(c_{\iota_k},\iota_k)$. 
With the above notations, we have $\delta(c)=\sum_k c'_k+c''_k$. Each $c''_k$ is supported on indices $\alpha:[k]\to[n]$, so the supports of the $c''_k$ are disjoint, and $c$ is projectable if and only if each of the $c''_k$ vanishes. The result follows.
\end{proof}



Given $R:G\action E=\bigoplus_{n=0}^NE_n$ a representation up to homotopy, we will define a map 
$$\lambda:C(G,E^*)\to C_{lin}(G\ltimes_RE)$$
from the $E^*$-valued cochains into the linear cochains of the semi-direct product. This is going to be a chain map and it will preserve the bigrading. Given $i,j$, write $T_{i,j}:G_{i+j}\to G_i$ for the map $T_{i,j}(g)=g^{-1}\iota_i$ that discards the first $j$ arrows and inverts the others:
$$(g_{i+j},\dots,g_1)\in G_{i+j}\quad\overset{T_{i,j}}\mapsto\quad (g_{j+1}^{-1},\dots,g_{i+j}^{-1})\in G_i$$
Since $t\circ T_{i,j}=x_j$, there is a vector bundle isomorphism $T_{i,j}^*t^*E_j=x_j^*E_j$. 
Given $c\in C^{i,j}(G,E^*)=\Gamma(G_i,t^*E^*_j)$, we define $\lambda(c)=(T_{i,j}^*c, \iota_j)\in C_{lin}^{i,j}(G\ltimes_R E)$.
By definition, if $(g,v)\in(G\ltimes_RE)_{i+j}$, then $\lambda(c)(g,v)=c(g^{-1}\iota_i)(v_{\iota_j})$.

\begin{proposition}\label{prop:lambda-morphism}
The morphism $\lambda:C(G,E^*)\to C_{lin}(G\ltimes_R E)$ is injective, its image is the complex of projectable cochains, and it commutes with the differentials: 
$$C(G,E^*)\overset{\lambda}\cong C_{proj}(G\ltimes_RE) \qquad H(G,E^*)\overset{\lambda}\cong H_{proj}(G\ltimes_RE)$$
\end{proposition}

\begin{proof}
Let $c\in C^{i,j}(G.E^*)=\Gamma(G_i,E^*_j)$.
The map $T_{i,j}:G_{i+j}\to G_i$ is a surjective submersion, and two simplices $g,g'\in G_{i+j}$ are in the same fiber if and only if they only differ on the first $j$-arrows. It follows from Lemma \ref{lemma:characterization-projectable} that $\lambda(c)=(T_{i,j}^*c,\iota_j)$ is projectable. Conversely, if $c'\in \Gamma(G_{i+j},x_j^*E^*_j)$ is projectable, it does not depend on the first $j$ arrows, and is a pullback $c'=T_{i,j}^*(c)$ of a section $c\in\Gamma(G_i,t^*E^*_j)$. This proves that $\lambda:C(G,E^*)\to C_{proj}(G\ltimes_RE)$ is an isomorphism. It remains to check the compatibility with the differentials. Given $c\in C^{i,j}(G,E^*)$, $(g,v)\in(G\ltimes E)_{i+j+1}$ and $m\geq0$, we need to show that $\lambda ( d_m(c))(g,v)=\delta_m(\lambda(c))(g,v)$. The computation is a straightforward application of the formulas in Lemmas \ref{lemma:d-linear} and \ref{lemma:d-dual}.
\end{proof}

\begin{remark}
Projectable cochains are defined intrinsically for any VB-groupoid, but we do not have a similar definition for higher vector bundles. This may be related with the fact that not every higher vector bundle is a semi-direct product, see \cite{dht}. Also, the use of the inversion in our formula for $\lambda$ would not be available in the more general setting of representations up to homotopy of higher Lie groupoids.
\end{remark}



\section{Filtrating by the regularity of the indices}


Let $G$ be a Lie groupoid, $E=\bigoplus_{n\geq0}E_n$ a (bounded) graded vector bundle over $G_0$, and $R:G\action E$ a representation up to homotopy. The main result of this section is Proposition \ref{prop:proj-lin}, proving that the inclusion 
$C_{proj}(G\ltimes E)\to C_{lin}(G\ltimes E)$
is a quasi-isomorphism, the b') of the introduction. The proof, which is quite technical, is inspired by the two-term case, originally proven in \cite[Lemma 3.1]{cd}.

\medskip


Recall that  $C^n_{lin}(G\ltimes E)=\bigoplus_{[k] \overset{\alpha}\to [n]}\Gamma(G_n,x^*_{\alpha(k)}E^*_k)$, where the sum is over the injective maps $\alpha:[k]\to [n]$ preserving 0. 
If $c\in C^n_{lin}(G\ltimes E)$, we write $c=\sum_\alpha (c_\alpha,\alpha)$, where $c_\alpha\in \Gamma(G_n,x^*_{\alpha(k)}E^*_k)$ for each $\alpha$.  We say that an index $\alpha$ is {\bf $m$-regular} if $\alpha(i)=i$ for every $i\leq \min(k,m)$. If $m\geq k$ this means that $\alpha=\iota_k:[k]\to[n]$ is the inclusion, and if $m\leq k$ this means that $\alpha(m)=m$, so $\alpha$ does not have a gap before $m+1$.
Every $\alpha$ is $0$-regular, and if $m\geq n-1$, then the only $m$-regular $\alpha$ are the regular ones of the previous section, namely the inclusions $\iota_k:[k]\to[n]$.

\begin{example}
When $n=4$, there are 16 indices $\alpha:[k]\to[4]$, which we can identify with subsets of $\{4,3,2,1,0\}$ containing $0$. We classify them according to their maximal regularity:
\begin{center}
\begin{tabular}{ |r|r|r|r| } 
 \hline
 3-regular & 2-regular & 1-regular & 0-regular \\ 
 \hline
 0 & 4210 & 310 & 20\\ 
 10 &  & 410  & 30\\ 
 210 & & 4310 & 40 \\ 
 3210 & & & 320 \\ 
 43210 & & & 420 \\ 
      & & & 430 \\ 
      & &   & 4320\\
 \hline
\end{tabular}
\end{center}
\end{example}


Write $C=C_{lin}(G\ltimes_RE)$ for short. 

\begin{definition}\label{def:m-projectable}
We say that $c\in C^n$ is {\bf $m$-regular} if it is supported on $m$-regular indices, and we say that $c$ is {\bf $m$-projectable} if both $c$ and $\delta(c)$ are $m$-regular. We write $F_mC^n\subset C^n$ for the subspace of $m$-regular $n$-cochains, and $\hat F_mC\subset C$ for the subcomplex of $m$-projectable cochains. 
\end{definition}


Note that $c\in F_mC^n$ if and only if $c(e,\alpha,g)=0$ for every $\alpha:[k]\to[n]$ that is not $m$-regular. 
Generalizing Lemma \ref{lemma:characterization-projectable}, we can characterize the $m$-projectable cochains as the $m$-regular cochains satisfying the following equivariant condition.

\begin{lemma}\label{lemma:characterization-m-projectable}
Let $c=\sum_{\substack{\alpha:[k]\to[n]\\ \alpha\text{ $m$-regular}}}(c_\alpha,\alpha)\in F_mC^n$. 
Then
$c\in \hat F_mC^n$, or equivalently $\delta(c)\in F_mC^{n+1}$, if and only if for every $\alpha$ the section
$c_\alpha$ does not depend on the first $\min(k,m)$ arrows. 
\end{lemma}

\begin{proof}
The proof is completely analogous to that of Lemma \ref{lemma:characterization-projectable}. Given $c=(c,\alpha)$ an homogeneous $m$-regular $n$-cochain, so $\alpha:[k]\to[n]$ is $m$-regular, using the formulas in Lemma \ref{lemma:d-linear}, we see that $\delta(c,\alpha)=c'+c''$ with $c'\in F_mC$, and $c''$ is given by
$$c''(g,v) = \sum_{r=1}^{\min(k,m)}(-1)^r c(d_rg)(v_{\delta_r\alpha})+
\sum_{r=1}^{\min(k,m)}(-1)^{r-1} c(d_0g)(v_{\alpha'\delta_r}).$$
Note that $\delta_r\alpha=\alpha'\delta_r$ for each $1\leq r\leq \min(k,m)$. Then $c$ is $m$-projectable if and only if $c''=0$. The same arguments as in Lemma \ref{lemma:characterization-projectable} show that $c''=0$ if and only if $c$ does not depend on the first $\min(k,m)$ arrows. 
For non-homogeneous cochains, performing the above decomposition for each term $(c_\alpha,\alpha)$, we get $\delta(c)=c'+\sum_\alpha c''_\alpha$ with $c'\in F_mC$.
Moreover, since $\supp c''_\alpha=\{\delta_r\alpha:1\leq r\leq \min(k,m)\}$, we have that $\supp c''_\alpha\cap \supp c''_{\beta}=\emptyset$ whenever $\alpha,\beta$ are $m$-regular, namely these supports are pairwise disjoint. Then $c$ is $m$-projectable if and only if $c''_\alpha=0$ for all $\alpha$, or equivalently, if $c_\alpha$ does not depend on the first $k$ arrows for every $\alpha$. 
\end{proof}


We have $\hat F_0C^n=F_0C^n=C_{lin}^n(G\ltimes_RE)$ for every $n\geq 0$.
And if $m\geq n$, then $c\in\hat F_mC^n$ if and only if $c$ and $\delta(c)$ are supported on the inclusions $\iota_k$, so $\hat F_mC^n=C^n_{proj}(G\ltimes_RE)$. 
This way we have a decreasing filtration by subcomplexes of $C_{lin}(G\ltimes E)$ which stabilizes on each degree:
$$\hat F_0C=C_{lin}(G\ltimes E)\supset \hat F_1C\supset\dots\supset
\hat F_mC\supset\dots\supset C_{proj}(G\ltimes E)$$
The main result of this section is the following.

\begin{proposition}
\label{prop:proj-lin}
Each of the above inclusions is a quasi-isomorphism. In particular, the inclusion $C_{proj}(G\ltimes_RE)\to C_{lin}(G\ltimes_RE)$ is a quasi-isomorphism.
\end{proposition}

Let us fix $m$ for the rest of the section.
We need to show that the inclusion $\hat F_{m+1}C\to \hat F_mC$ is a quasi-isomorphism.  
This proof is quite technical, involves several definitions and computations, and will demand the next few pages.



For each $n>m$, write $\eta:[n+1]\to[n]$ for the function given by $\eta(i)=i$ if $i<m$, $\eta(m)=n$, and $\eta(i)=i-1$ if $i>m$. We extend this definition to $n\leq m$ by setting $\eta=\sigma_n:[n+1]\to[n]$, though we are mainly interested in the case $n>m$. We define a vector bundle map
$$h:(G\ltimes E)_n\to (G\ltimes E)_{n+1}
\qquad
h(e,\alpha,g)=\begin{cases}
(-1)^{m+1}(e,\delta_{m+1}\alpha,g\eta) & \alpha(k)>m\\
(0,g\eta) & \alpha(k)\leq m
\end{cases}$$
Note that when $n>m$ the base map $G_{n}\to G_{n+1}$, $g\mapsto g\eta$, makes sense for $G$ has an inversion. Thus, if $g=(g_n,\dots,g_1)$, then $g\eta=(g_n,\dots,g_{m+1},(g_n\cdots g_{m+1})^{-1},(g_n\dots g_m),g_{m-1},\dots,g_1)$.

The above vector bundle map yields a degree -1 operator $h^*:C^{n+1}_{lin}(G\ltimes_RE)\to C^{n}_{lin}(G\ltimes_RE)$, $h^*(c)(v)=c(h(v))$, 
which we use to twist the identity of $C=C_{lin}(G\ltimes_RE)$, so as to get a {\bf regularization map} $I=\id+h^*\delta+\delta h^*:C\to C$, a chain map homotopic to the identity:
$$I:C^n\to C^n \qquad I(c)=c+ h^*\delta(c)+\delta h^*(c)=c+ 
\sum_{r=0}^{n+1} (-1)^r c d_r h + \sum_{r=0}^n (-1)^r c h d_r$$

Iterations of $I$ will play the role of a homotopy inverse for the inclusion $\hat F_{m+1}C\to\hat F_{m}C$. As a preparation, the next lemma provides identities relating $h$, the face maps $d_r$, and the $m$-regular cochains $c$.



\begin{lemma}\label{lemma:H}
Let $c\in F_mC^n$ be an $m$-regular cochain.
\begin{enumerate}[a)]
	\item If $m+1<r$ then $c h d_r \in F_mC^n$ and if $m<r$ then $c d_r h\in F_mC^n$.
    \item If $m<r<n$ then $h d_r=d_{r+1} h$, and therefore $ch d_r = c d_{r+1} h$.
    \item If $0< r\leq m$ then $cd_rh =0$.
    \item If $0< r\leq m$ and $ch d_r(e,\alpha,g)\neq0$ then $n-1>m$, 
	$[m+1]\setminus\alpha=\{r\}$ and 
	$$ch d_r(e,\alpha,g)=(-1)^{m+1}c(e,\delta_{m+1}\sigma_{r-1}\alpha,g\delta_r\eta).$$
    \item  If $\alpha$ is not $m+1$-regular and $cd_0 h(e,\alpha,g)\neq 0$ then $n-1>m$, $\alpha$ is $m$-regular and 
$$c d_0 h(e,\alpha,g)= -c(e,\alpha,g\eta\delta_0).$$
    \item If $\alpha$ is not $m+1$-regular and $c h d_0(e,\alpha,g)\neq 0$ then $n-1>m$, $[m+1]\setminus\alpha=\{r\}$  and
$$c h d_0(e,\alpha,g)=(-1)^{m+r}c(e,\delta_{m+1}\sigma_{r-1}\alpha,g\delta_0\eta).$$
\end{enumerate}
\end{lemma}

\begin{proof}
We repeatedly use the formulas for the face maps of $G\ltimes_RE$ reviewed in Remark \ref{rmk:homogeneous-vectors}.
\begin{enumerate}[a)]
\item 
Consider $(e,\alpha,g)\in C^n$ with $\alpha$ not $m$-regular. Then $h d_r(e,\alpha,g)$ is either $0$ or supported on $\delta_{m+1}\sigma_{r-1}\alpha$, which is not $m$-regular, we have $ch d_r(e,\alpha,g)=0$. This proves that $ch d_r\in F_mC^n$. The proof that $cd_rh\in F_mC^n$ when $m<r$ is analogous.

\item Given $(e,\alpha,g)\in C^n$, if either $r\in\alpha$ or $\alpha(k)\leq m$ then $hd_r(e,\alpha,g)=0=d_{r+1}h (e,\alpha,g)$. If $r\notin\alpha$ and $\alpha(k)>m$ then
$h d_r(e,\alpha,g)=(-1)^{m+1}(e,\delta_{m+1}\sigma_{r-1}\alpha,g\delta_r\eta)$, which equals $d_{r+1}
h(e,\alpha,g)=(-1)^{m+1}(e,\sigma_{r}\delta_{m+1}\alpha,g\eta\delta_{r+1})$.

\item Let $(e,\alpha,g)\in C^n$. If $\alpha(k)\leq m$ then $h(e,\alpha,g)=0$. If $n\geq\alpha(k)>m$ then $d_r h(e,\alpha,g)$ is either 0 or supported on $\sigma_{r-1}\delta_{m+1}\alpha$, which is not $m$-regular, for it does not contain $m$. Thus $cd_r h(v,\alpha,g)=0$. 

\item Since $hd_r(e,\alpha,g)\neq 0$ we have that $r\notin\alpha$ and that $n-1\geq\sigma_{r-1}\alpha(k)>m$. Then $hd_r(e,\alpha,g)$ is supported on $\delta_{m+1}\sigma_{r-1}\alpha$, which has to be $m$-regular, from where $\sigma_{r-1}\alpha\supset[m]$, or equivalently, $[m+1]\setminus\alpha=\{r\}$. The formula follows from those of $d_r$ and $h$.

\item Since $h(e,\alpha,g)\neq0$ we have $n\geq\alpha(k)>m$ and $h(e,\alpha,g)=(-1)^{m+1}(e,\delta_{m+1}\alpha,g\eta)$. 
Suppose first that $m=0$. Since $\alpha$ is not 1-regular, we have $n>1$, $\delta_1\alpha$ is not 1-regular, and $d_0h(e,\alpha,g)= -d_0 (e,\delta_{1}\alpha,g\eta) =-(e,\alpha,g)$.
Suppose now that $m>0$. The support of $d_0h (e,\alpha,g)$ is contained in 
$$\{\beta:\beta'\iota_k=\delta_{m+1}\alpha\}\cup\{\beta:\beta'\delta_i=\delta_{m+1}\alpha\}$$
If $\alpha$ is not $m$-regular then none of these $\beta$ is $m$-regular, contradicting $cd_0h(e,\alpha,g)\neq 
0$. Then $\alpha$ is $m$-regular, and since it is not $m+1$-regular, $n-1>m$. The only $m$-regular $\beta$ is the one with $\beta'\delta_{m+1}=\delta_{m+1}\alpha$, and we can conclude that
$cd_0 h(e,\alpha,g)=(-1)^{m+1} c(d_0(e,\delta_{m+1}\alpha,g\eta))= -c(e,\alpha,g\eta\delta_0)$.

\item The proof is similar to the previous one. Since $hd_0(e,\alpha,g)\neq 0$ and $d_0(e,\alpha,g)\in C^{n-1}$, then $n-1>m$. Suppose first that $m=0$. Since $\alpha$ is not 1-regular, then $[1]\setminus\alpha=\{1\}$, so $r=1$, and $chd_0(e,\alpha,g)=ch(e,\sigma_0\alpha,g\delta_0) = -c(e,\delta_{1}\sigma_0\alpha,g\delta_0\eta)$. Suppose now that $m>0$. The support of $hd_0(e,\alpha,g)$ is contained in
$$\{\delta_{m+1}\beta:\beta'\iota_k=\alpha\}\cup\{\delta_{m+1}\beta:\beta'\delta_i=\alpha\}$$
If $[m+1]\setminus\alpha$ contains at least two elements then none of these $\beta$ is $m$-regular, contradicting $chd_0(e,\alpha,g)\neq0$. Then $[m+1]\setminus\alpha=\{r\}$ for some $r$. The only $m$-regular $\beta$ is the one with $\beta'\delta_r=\alpha$, which can be written as $\delta_{m+1}\sigma_{r-1}\alpha$, and we can conclude that
$ch d_0(e,\alpha,g)=(-1)^{m+r}c(e,\delta_{m+1}\sigma_{r-1}\alpha,g\delta_0\eta)$.
\end{enumerate}
\end{proof}



We now use Lemma \ref{lemma:H} to prove that the regularization map $I$ preserves the filtration $\hat F_mC$.

\begin{lemma}\label{lemma:I-filtration}
$I(\hat F_mC)\subset F_mC$, and therefore $I(\hat F_mC)\subset \hat F_mC$.
\end{lemma}

\begin{proof}
Let $c\in \hat F_mC^n$, and let $(e,\alpha,g)\in (G\ltimes_RE)_n$ with $\alpha$ not $m$-regular. Let us prove that $I(c)(e,\alpha,g)=0$, from where $I(c)\in F_mC^n$. By definition,
$$I(c)(e,\alpha,g)=c(e,\alpha,g)+ 
\sum_{r=0}^{n+1} (-1)^r c d_rh(e,\alpha,g) + \sum_{r=0}^n (-1)^r c hd_r(e,\alpha,g).$$
Note that $c(e,\alpha,g)=0$ because $c\in F_mC^n$.
The first sum vanishes by Lemma \ref{lemma:H} a), c) and e).
We need to show that the second sum also vanishes. 
If $n-1\leq m$ then $h(v)=0$ for every $v\in (G\ltimes_RE)_{n-1}$ and we are done. Suppose then that $n-1>m$. We consider two cases.

$\bullet$ Case $[m+1]\setminus\alpha$ has more than one element: Then every term in the second sum vanishes by Lemma \ref{lemma:H} a), d) and f), except perhaps for $chd_{m+1}(e,\alpha,g)$. If either $m+1\in\alpha$ or $\sigma_{m}\alpha(k)\leq m$ then $hd_{m+1}(e,\alpha,g)=0$. If $m+1\notin\alpha$ and 
$\sigma_{m}\alpha(k)> m$ then $hd_{m+1}(e,\alpha,g)$ is supported on $\delta_{m+1}\sigma_m\alpha$, which is not $m$-regular, so $chd_{m+1}(e,\alpha,g)=0$.

$\bullet$ Case $[m+1]\setminus\alpha=\{r\}$ consists of a single element: Then $r\leq m$ because $\alpha$ is not $m$-regular. Since $m+1\in\alpha$ we have $d_{m+1}(e,\alpha,g)=0$. By Lemma \ref{lemma:H} a), d) and f) there are only two surviving terms, those corresponding to $0$ and $r$, so
$$I(c)(e,\alpha,g)=(-1)^{m+r+1}c(e,\delta_{m+1}\sigma_{r-1}\alpha,g\delta_r\eta)+(-1)^{m+r}c(e,\delta_{m+1}\sigma_{r-1}\alpha,g\delta_0\eta)$$
Since $c\in\hat F_mC^n$, this vanishes by Lemma \ref{lemma:characterization-m-projectable}, for $g\delta_r\eta$ and $g\delta_0\eta$ only differ on the first $m$ arrows.

Starting with $c\in \hat F_mC^n$, we showed that $I(c)\in F_mC^n$, but
$\delta(c)\in \hat F_mC^{n+1}$ too, and the same argument as above proves that $I\delta(c)=\delta I(c)\in F_mC^{n+1}$. We conclude that $c\in \hat F_mC^n$.
\end{proof}

Next, we show that iterations of the regularization map $I$  increase the regularity. To that end, given $c\in F_mC^n$ an $m$-regular $n$-cochain, 
we define its {\bf defect} as 
$$\ell(c)=\begin{cases}\max\{\alpha(m+1)-(m+1): c_\alpha\neq 0\} & c\notin F_{m+1}C^n\\ 0 & c\in F_{m+1}C^n\end{cases}$$
The defect $\ell(c)$ is the size of the largest gap after $m$ appearing in the support of $c$. Note that $0\leq \ell(c)\leq n-1$, and that $\ell(c)=0$ if and only if $c\in F_{m+1}C^n$. Thus, we can think of the defect $\ell(c)$ as measuring how far is $c\in F_mC^n$ from being in $F_{m+1}C^n$. By the previous Lemma \ref{lemma:I-filtration}, if $c\in\hat F_mC^n$, then $I(c)\in\hat F_mC^n$ and its defect $\ell(I(c))$ is also well-defined.

\begin{proposition}\label{prop:defect}
Given $c\in \hat F_mC^n$, the defect $\ell(c)$ decreases with $I$, namely $\ell(I(c))<\ell(c)$ if $\ell(c)>0$, and $\ell(I(c))=0$ if $\ell(c)=0$.
\end{proposition}

\begin{proof}

If $m\geq n-1$ then $\hat F_mC^n=\hat F_{m+1}C^n$ and therefore $\ell(c)=\ell(I(c))=0$. Suppose then that $m<n-1$. 
Given $(e,\alpha,g)$ with $\alpha:[k]\to[n]$ an $m$-regular index such that $\alpha(m+1)-m+1\geq\max\{\ell(c),1\}$, we need to show that $I(c)(e,\alpha,g)=0$. 
We have
$$I(c)(e,\alpha,g)=c(e,\alpha,g)+ 
\sum_{r=0}^{n+1} (-1)^r c d_rh(e,\alpha,g) + \sum_{r=0}^n (-1)^r c hd_r(e,\alpha,g).$$
By Lemma \ref{lemma:H} b), c) and d) several cancellations and vanishing occur, and the above sum reduces to only six terms, namely:
\begin{align*}
I(c)(e,\alpha,g) & = 
c(e,\alpha,g)+ cd_{0}h(e,\alpha,g)+\\
& \qquad \qquad + (-1)^{m+1}cd_{m+1}h(e,\alpha,g)+chd_0(e,\alpha,g)+\\
& \qquad \qquad \qquad \qquad +
(-1)^{n+1}cd_{n+1}h(e,\alpha,g)+
 (-1)^nchd_n(e,\alpha,g)
\end{align*}


We claim that each of these lines vanishes.
Regarding the first line, by Lemma \ref{lemma:H} e) we have that $cd_0h(e,\alpha,g)= -c(e,\alpha,g\eta\delta_0)$, and this cancels out with $c(e,\alpha,g)$, for $g\eta\delta_0$ and $g$ only differ on the first $m$ arrows, see Lemma \ref{lemma:characterization-m-projectable}.
The argument for the second line is analogous: by Lemma \ref{lemma:H} f), with $r=m+1$, we have $chd_0(e,\alpha,g)= -c(e,\delta_{m+1}\sigma_{m}\alpha,g\delta_0\eta)= -c(e,\alpha,g\delta_0\eta)$, and this cancels out with 
$(-1)^{m+1}cd_{m+1}h(e,\alpha,g)=c(e,\sigma_m\delta_{m+1} \alpha,g\eta\delta_{m+1})=
c(e,\alpha,g\eta\delta_{m+1})$, 
for $g\delta_0\eta$ and $g\eta\delta_{m+1}$ only differ on the first $m$ arrows.
Finally, each term in the third line vanishes. If $n\in\alpha$ then $d_{n+1}h(e,\alpha,g)$ and $hd_n(e,\alpha,g)$ vanish, 
and if $n\notin\alpha$ then $d_{n+1}h(e,\alpha,g)$ and $hd_n(e,\alpha,g)$ are both supported on $\sigma_n\delta_{m+1}\alpha$, which has a gap after $m$ larger than $\ell(c)$.
\end{proof}


We have everything in place to show the main result of the section, 
Proposition \ref{prop:proj-lin}, which is an immediate corollary of the following.

\begin{proposition}
The inclusion $\hat F_{m+1}C\to \hat F_mC$ is a quasi-isomorphism.
\end{proposition}

\begin{proof}
Fix $n\geq 0$. We want to show that $H^n(\hat F_{m+1}C)\to H^n(\hat F_{m}C)$ is an isomorphism. Given $c\in \hat F_{m}C^n$, the defects of $\ell(c)$ and of $\ell(\delta(c))$ are bounded by $n-1$ and $n$, respectively. Then, by Proposition \ref{prop:defect}, $I^n(c)\in \hat F_{m+1}C^n$ is $m+1$-projectable. The result easily follows from here: If $c\in \hat F_mC^n$ is closed, then $[c]=[I^n(c)]$, and since $I^n(c)\in \hat F_{m+1}C^n$, $H^n(\hat F_{m+1}C)\to H^n(\hat F_mC)$ is an epimorphism; If $c\in \hat F_{m+1}C$ is exact, so $c=\delta(c')$ for some $c'\in \hat F_{m}C^{n-1}$, then $[c]=[I^n(c)]=[\delta(I^n(c'))]$, and since $I^n(c')\in \hat F_{m+1}C^{n-1}$, $H^n(\hat F_{m+1}C)\to H^n(\hat F_mC)$ is a monomorphism. 
\end{proof}


\section{Cohomology of higher Lie groupoids}

This final section, which can be read independently, reviews the basic notions on higher Lie groupoids and hypercovers, following \cite{henriques,zhu,getzler,wolfson}. Our treatment provides some simplifications and improvements, as in Proposition \ref{prop:hom} on the smoothness of spaces of diagrams, Lemma \ref{lemma:nice-family} on the base-change and composition of hypercovers, and Theorem \ref{thm:higher-descent}, the main result, on the invariance of cohomology by hypercovers, for which we provide a geometric proof. As a simple corollary we get the claim c') of the introduction, the last piece of the puzzle.

\medskip


Given $X$ a simplicial manifold, we write $X^K=\hom(K,X)$ for the set of simplicial maps $K\to X$. We say that $X^K$ is {\bf smooth} if the limit $X^K=\lim_{\Delta^n\xto\alpha K}X_n$, which always exists in $\Top$, also exists in $\Man$, and it has the right topology. 
If $K$ is spanned by the simplices $\sigma_i\in K_{n_i}$, $i=1,\dots,r$, then $X^K$ is smooth if and only if $X^K\to \prod_{i=1}^r X_{n_i}$, $\phi\mapsto (\phi(\sigma_i))_i$, is an embedded submanifold. 
Note that $X^{\Delta^n}=X_n$ is always smooth. 
The {\bf horn-space} $X^{\Lambda^n_k}$ is the space of maps $\Lambda^n_k\to X$, and the {\bf matching-space} $X^{\partial\Delta^n}$ is the space of maps $\partial\Delta^n\to X$. Since $\Lambda^0_0=\partial\Delta^0=\emptyset$ we have $X^{\Lambda^0_0}=X^{\partial\Delta^0}=\ast$.

\begin{definition}
A {\bf higher Lie groupoid} $X$ is a simplicial manifold such that the horn spaces $X^{\Lambda^n_k}$ are smooth and the restriction maps $d_{n,k}:X_n=X^{\Delta^n}\to X^{\Lambda^n_k}$ are surjective submersions. It is {\bf $m$-strict} or a {\bf Lie $m$-groupoid} if $d_{n,k}$ is bijective when $n>m$. It is {\bf acyclic} if the matching spaces $X^{\partial\Delta^n}$ are smooth and the restriction maps $\partial_n:X_n\to X^{\partial\Delta^n}$ are surjective submersions. 
\end{definition}



\begin{example}\begin{enumerate}[a)]
    \item A Lie 0-groupoid $M$ is a constant simplicial manifold. A morphism $f:M\to N$ between Lie 0-groupoids is the same as a smooth map. $M$ is acyclic if and only if $M=\ast$.  
    \item A Lie 1-groupoid $G$ is the nerve of a Lie groupoid, as classically defined, see e.g. \cite{bdh}. A morphism of Lie 1-groupoids is the same as a morphism between the corresponding Lie groupoids. A Lie groupoid is acyclic if and only if it is a pair groupoid.
    %
    \item The Dold-Kan construction on a chain complex $((E_n)_{n\geq 0},\partial)$ of vector bundles over a manifold $M$ is a higher Lie groupoid, it is a Lie $m$-groupoid if $E_n=0_M$ for every $n>m$, and it is acyclic if its homology is trivial.
\end{enumerate}
\end{example}


Given $q:X\to Y$ a map of simplicial manifolds, and given $i:K\to L$ an inclusion of simplicial sets, we write $X^K\times_{Y^K}Y^L$ for the space of commutative diagrams as below, and we say that it is smooth if this iterated limit, which always exists in $\Top$, also exists in $\Man$, and it has the right topology.
$$\xymatrix{K \ar[r]^\phi \ar[d]_i & X \ar[d]^q \\ L \ar[r]^\psi & Y}$$
If $K$ is spanned by
the simplices $\sigma_i\in K_{n_i}$, $i=1,\dots,r$, and $L$ is spanned by the $\sigma_i$ and the $\tau_j\in L_{n_j}\setminus K$, $j=1,\dots,s$, then $X^K\times_{Y^K}Y^L$ is smooth if and only if  $X^K\times_{Y^K}Y^L\to \prod_{i=1}^r X_{n_i}\times\prod_{j=1}^s Y_{n_j}$, $(\phi,\psi)\mapsto (\phi(\sigma_i)_i,(\tau_j)_j)$,  is an embedded submanifold. The space $X^K\times_{Y^K}Y^L$ can be smooth even when $X^K$, $Y^L$ or $Y^L$ are not. 
The {\bf relative horn-space} is $X^{\Lambda^n_k}\times_{Y^{\Lambda^n_k}}Y_n$ and the {\bf relative matching space} is $X^{\partial \Delta^n}\times_{Y^{\partial\Delta^n}}Y_n$. 


\begin{definition}
A {\bf fibration} $f:X\to Y$ is a map of simplicial manifolds such that the relative horn spaces $X^{\Delta^n_k}\times_{Y^{\Delta^n_k}}Y_n$ are smooth and the restriction maps $d_{n,k}:X_n\to X^{\Delta^n_k}\times_{Y^{\Delta^n_k}}Y_n$ are surjective submersions. It is {\bf $m$-strict} if $d_{n,k}$ is bijective when $n>m$. A {\bf hypercover} is a map such that the relative matching spaces $X^{\partial \Delta^n}\times_{Y^{\partial\Delta^n}}Y_n$ are smooth and the restriction maps $\partial_n:X_n\to X^{\partial \Delta^n}\times_{Y^{\partial\Delta^n}}Y_n$ are surjective submersions.
\end{definition}


Since $\partial\Delta^0=\Lambda^0_0=\emptyset$, we have 
$X^{\Delta^0_0}\times_{Y^{\Delta^0_0}}Y_0= 
X^{\partial\Delta^0}\times_{Y^{\partial\Delta^0}}Y_0=Y_0$, and the restriction maps $d_{0,0}$, $\partial_0$ identify with $f_0:X_0\to Y_0$. 
A simplicial manifold $X$ is a higher Lie groupoid if and only if the projection $X\to\ast$ is a fibration, and it is acyclic if and only if $X\to\ast$ is a hypercover. 
Set-theoretically, a fibration is a Kan fibration, and a hypercover is an acyclic fibration.
Our fibrations are called {\it stacks} in \cite{wolfson}, with a slight difference, they require $n>0$, so in their stacks $f_0$ need not be a submersion. 


\begin{example}\begin{enumerate}[a)]
    \item A map $f:M\to N$ between smooth manifolds, viewed as a morphism of Lie 0-groupoids, is a fibration if and only if $f$ is a surjective submersion, and it is a hypercover if and only if $f$ is a diffeomorphism.
    \item A Lie groupoid morphism $f:G\to H$, viewed as a morphism between their nerves, is a fibration if and only if it is so in the usual sense, and it is a hypercover if and only if it is also Morita, see e.g. \cite{bdh,dh}.
    \item A chain map $f:E\to E'$ of complexes of vector bundles over $M$ yields a fibration if and only if $f_n:E_n\to E'_n$ is an epimorphism for all $n$, and it is a hypercover if moreover is a quasi-isomorphism.
\end{enumerate}
\end{example}



The following key result has some similar versions in the literature, see \cite[\S 2]{henriques} and \cite[\S 2]{wolfson}. 
Given $L$ a simplicial set, recall that $K\subset L$ is an elementary collapse of $L$, notation $L\searrow^e K$, if there is an isomorphism $L\cong K\cup_{\Lambda^n_k}\Delta^n$, and that $K\subset L$ is a collapse of $L$, notation $L\searrow K$, if $K$ can be obtained from $L$ by a sequence of elementary collapses. For instance $\Delta^n\searrow^e \Lambda^n_k\searrow\ast$.

\begin{proposition}\label{prop:hom}
Let $i:K\to L$ be an inclusion of finite simplicial sets and $f:X\to Y$ a map of simplicial manifolds. Suppose that $Y^L$ is smooth and that one of the following holds:
\begin{enumerate}[a)]
    \item $f$ is a hypercover, or
    \item $f$ is a fibration and $L\searrow K\searrow \ast$.
\end{enumerate}
Then $X^K\times_{Y^K}Y^L$ and $X^L$ are smooth, and $X^L\to X^K\times_{Y^K}Y^L\to Y^L$ are surjective submersions.
\end{proposition}

\begin{proof}
We consider a decreasing  filtration as follows: 
$$L=F_a\supset F_{a-1}\supset\dots\supset F_b=K\supset\dots\supset F_0=\ast\supset F_{-1}=\emptyset$$ 
In case a), we demand each step to be the adjunction of a cell, namely $F_r=F_{r-1}\cup_{\partial\Delta^{n_r}}\Delta^{n_r}$. 
In case b), we demand each step to be an elementary collapse, namely $F_r=F_{r-1}\cup_{\partial\Lambda^{n_r}_{k_r}}\Delta^{n_r}$.
This leads to a tower of spaces
$$X^L=Y^L\times_{Y^{F_{a}}}X^{F_{a}} \to Y^L\times_{Y^{F_{a-1}}}X^{F_{a-1}}\to\dots\to 
Y^L\times_{Y^K}X^K\to \dots\to Y^L\times_{Y^{F_{-1}}}X^{F_{-1}}=Y^L$$
Each step in the tower is the base-change of a surjective submersion, in case a) of the map $X_{n_r}\to X^{\partial\Delta^{n_r}}\times_{Y^{\partial\Delta^{n_r}}}Y_{n_r}$, in case b) of the map $X_{n_r}\to X^{\Lambda^{n_r}_{k_r}}\times_{Y^{\Lambda^{n_r}_{k_r}}}Y_{n_r}$:
$$\xymatrix{ 
Y^L\times_{Y^{F_r}}X^{F_r} \ar[r] \ar[d] & 
Y^L\times_{Y^{F_{r-1}}}X^{F_{r-1}} \ar[d] \\
X_{n_r} \ar[r] &X^{\partial\Delta^{n_r}}\times_{Y^{\partial\Delta^{n_r}}}Y_{n_r}}
\qquad 
\xymatrix{ 
Y^L\times_{Y^{F_r}}X^{F_r} \ar[r] \ar[d] & 
Y^L\times_{Y^{F_{r-1}}}X^{F_{r-1}} \ar[d] \\
X_{n_r} \ar[r] & X^{\Lambda^{n_r}_{k_r}}\times_{Y^{\Lambda^{n_r}_{k_r}}}Y_{n_r}}$$
It follows inductively that $Y^L\times_{Y^K}X^K$ and that $X^L$ are smooth, and that $X^L\to X^K\times_{Y^K}Y^L$ is a surjective submersion.
\end{proof}



\def\u{\underline}

\begin{corollary}
\begin{enumerate}[a)]
    \item A hypercover $f:X\to Y$ is always a  fibration.
    \item In a fibration every map $f_n:X_n\to Y_n$ is a surjective submersion.
    \item If $X$ is a higher Lie groupoid and $K\searrow \ast$ then $X^K$ is smooth. 
    \item If $X$ is a higher Lie groupoid and $L\searrow K\searrow\ast$ then $X^L\to X^K$ is a surjective submersion.
\end{enumerate}
\end{corollary}
 
\begin{proof}
a) follows by taking $L=\Delta^n$ and $K=\Lambda^n_k$. b) follows by taking $L=\Delta^n$ and $K=\ast$. c) follows by taking $Y=\ast$ and $K=\ast$. d) follows by using c) and taking $Y=\ast$. 
\end{proof}


\begin{remark}
It also follows from the proof of Proposition \ref{prop:hom} that if the relative horn spaces $X^{\Lambda^n_k}\times_{Y^{\Lambda^n_k}}Y_n$ are smooth and the restriction maps $d_{n,k}:X_n\to X^{\Lambda^n_k}\times_{Y^{\Lambda^n_k}}Y_n$ are surjective submersions for $n<m$, then the relative horn-spaces $X^{\Lambda^m_k}\times_{Y^{\Lambda^m_k}}Y_m$ are also smooth. This shows some redundancy in the definition of fibrations. The situation with hypercovers is analogous. See \cite[Corollary 2.5]{henriques}, \cite[Lemmas 2.10, 2.15]{wolfson} and \cite[Lemma 2.4]{zhu}.
\end{remark}


A {\bf differentiable $n$-stack} is the class of a higher Lie groupoid under hypercovers. More precisely, if $X,Y$ are higher Lie groupoids, we say that they are Morita equivalent if there is a third higher Lie groupoid and hypercovers $X\from Z\to Y$. The next Lemma shows that this is indeed an equivalence relation. A property of differentiable $n$-stacks is a property of Lie $n$-groupoids that is invariant under hypercovers.

\begin{lemma}\label{lemma:nice-family}
Fibrations and hypercovers between higher Lie groupoids are closed under both base-change and composition.
\end{lemma}
\begin{proof} 
{\it Base-change:}
Let $f:X\to Y$ be either a fibration or a hypercover, and let $K=\Lambda^n_k$ or $K=\partial\Delta^n$, respectively, and $L=\Delta^n$. Given $g:Y'\to Y$ a simplicial map, since each $X_n\to Y_n$ is a surjective submersion, the fibered product $X'=Y'\times_YX$ is a well-defined simplicial manifold. 
The space $Y'^L\times_{Y'^K}X'^K=Y'^L\times_{Y^L}(Y^L\times_{Y^K}X^K)$ is smooth, for the projection $Y^L\times_{Y^K}X^K\to Y^L$ is a surjective submersion by Proposition \ref{prop:hom}.
The restriction map $X'^L\to Y'^L\times_{Y'^K}X'^K$ identifies with
$Y'^L\times_{Y^L}(X^L)\to Y'^L\times_{Y^L}(Y^L\times_{Y^K}X^K)$, which is the base-change of a surjective submersion, again by Proposition \ref{prop:hom}.

{\it Composition:}
Again, let $f:X\to Y$ and $g:Y\to Z$ be either fibrations or hypercovers, and consider $K=\Lambda^n_k$ or $K=\partial\Delta^n$, respectively, while $L=\Delta^n$. 
Pick a filtration of $L$ as in Proposition \ref{prop:hom}.  This gives rise to a tower as follows, 
$$X^L\to \dots \to Y^L\times_{Y^K}X^K\to \dots \to Z^L\times_{Z^K}X^K\to\cdots\to Z^L\times_{Z^K}Y^K\to\cdots\to Z^L$$
We need to show that $Z^L\times_{Z^K}X^K$ is smooth and that $x^L\to Z^L\times_{Z^K}X^K$ is a surjective submersion.
Each of the four segments is decomposed in elementary steps, which are the base-change of a surjective submersion, which shows the result.
\end{proof}




The size of a hypercover can be measured by the fibers of the relative matching maps $\partial_n: X_n\to X^{\partial\Delta^n}\times_{Y^{\partial\Delta^n}}Y_n$. We say that a hypercover $f:X\to Y$ is {\bf $m$-simple} if $\partial_n$ is a diffeomorphism for every $n\neq m$. 

\begin{lemma}\label{lemma:simple}
If $f:X\to Y$ is an $m$-simple hypercover then $X_n\cong X^{\partial\Delta^n}\times_{Y^{\partial\Delta^n}}Y_n=Y_n$ for $n<m$, $\partial_m=f_m:X_m\to M_m(X/Y)\cong Y_m$, and for $n>m$ we have a fibered product
$$\xymatrix{ X_n \ar[r] \ar[d] & \prod X_m \ar[d] \\
Y_n \ar[r] & \prod Y_m}$$
\end{lemma}

\begin{proof}
For $n<m$ the claim follows inductively. Since $X_{m-1}\cong Y_{m-1}$, we have $M_m(X/Y)\cong Y_m$. The final claim is easy, and it follows for instance from the proof of the hom-proposition. 
\end{proof}


Given $X$ a simplicial set and $m\geq 0$, recall that  the {\bf coskeleton} $\cosk_mX$ is the simplicial set whose $n$-simplices are maps $(\cosk_mX)_n=X^{\sk_m\Delta^n}=\hom(\sk_m\Delta^n,X)$ \cite{gj}. 
A simplicial set $X$ is {\bf $m$-coskeletal} if $X=\cosk_mX$. 
If $X$ is an $m$-groupoid, then it is $m+1$-coskeletal. 
Given $f:X\to Y$ and $m$, the {\bf relative coskeleton} is the simplicial set $\cosk_m(X/Y)=\cosk_m X\times_{\cosk_m Y}Y$. Its $n$-simplices are 
$X^{\sk_m\Delta^n}\times_{Y^{\sk_m\Delta^n}}Y^{\Delta^n}$, or in other words, commutative squares
$$\xymatrix{
\sk_m\Delta^n \ar[r] \ar[d] & X \ar[d] \\
\Delta^n \ar[r] & Y
}$$
If $m\geq n$ then $\cosk_m(X/Y)_n=X_n$ and if $m=-1$ then $\cosk_m(X/Y)_n=Y_n$.
Note that a hypercover $f:X\to Y$ is $m$-simple if and only if $\cosk_m(X/Y)=X$ and $\cosk_{m-1}(X/Y)=Y$.
If $X$ and $Y$ are $m$-groupoid, they are $m+1$-coskeletal, and the {\bf coskeleton tower} is 
$$X=\cosk_{m+1}(X/Y)\to \cosk_{m}(X/Y)\to \dots\to \cosk_{1}(X/Y)\to \cosk_{0}(X/Y)\to \cosk_{-1}(X/Y)=Y$$


The following is yet another corollary of Proposition \ref{prop:hom}.

\begin{proposition}\label{prop:cosk}
If $f$ is a hypercover and $m\geq0$ then the relative coskeleton $\cosk_m(X/Y)_n$ is smooth, so $\cosk_m(X/Y)$ is a simplicial manifold.
    Moreover, each map in the relative coskeletal tower $r:\cosk_m(X/Y)\to\cosk_{m-1}(X/Y)$ is a simple hypercover.
\end{proposition}

\begin{proof}
The smoothness of $\cosk_m(X/Y)_n$ follows from Proposition \ref{prop:hom}, by taking $K=\sk_m\Delta^n$ and $L=\Delta^n$. Since $\cosk_m(X/Y)_n\subset \prod X_m\times \prod Y_n$ is embedded, the simplicial maps are smooth, and $\cosk_m(X/Y)$ is a well-defined simplicial manifold.
It remains to show that each step in the coskeletal tower is a hypercover.

A point in the relative matching space $\cosk_m(X/Y)^{\partial\Delta^n}\times_{\cosk_{m-1}(X/Y)
^{\partial\Delta^n}}\cosk_{m-1}(X/Y)_n$ is a diagram as below, 
$$\xymatrix@R=0pt{ & \sk_{m-1}\Delta^n \ar[rrd] \ar[ddd] \ar@/^/[rrrd] & & & \\
\sk_{m-1}\partial\Delta^n \ar[ru] \ar[rrd] \ar[ddd]& & & \sk_m\Delta^n \ar@{-->}[r] \ar[ddd]& X \ar[ddd] \\
& & \sk_m\partial\Delta^n \ar[ddd]\ar[ru] \ar@/_/[rru] & \\
& \Delta^n \ar[rrd] \ar@/^/[rrrd] & & & \\
\partial\Delta^n \ar[ru] \ar[rrd] & & & \Delta^n \ar@{-->}[r] & Y \\
& & \partial\Delta^n \ar[ru] \ar@/_/[rru]&
}$$
while a point in $\cosk_m(X/Y)_n$ is a dotted filling. If $n\neq m$ then $\sk_{m-1}\Delta^n\cup\sk_m\partial\Delta^n=\sk_m\Delta^n$ and every diagram has a unique filling, so the matching map $\partial_n$ is an isomorphism. And if $n=m$, the matching map $\partial_m$ identifies with $X_n\to X^{\partial\Delta^n}\times_{Y^{\partial\Delta^n}} Y_n$, hence the result.
\end{proof}





The next theorem shows that a hypercover gives an isomorphism in cohomology. This appears implicitly in \cite{getzler}. We present here an alternative geometric proof.

\begin{theorem}\label{thm:higher-descent}
If $f:X\to Y$ is a hypercover then $f^*:H^\bullet(Y)\to H^\bullet(X)$ is an isomorphism.
\end{theorem}

\begin{proof}
By the coskeleton tower of Proposition \ref{prop:cosk} we can assume that $f$ is simple. Fix $\mu^m$ an unital vertical density on $X_m\to Y_m$. This yields a system of vertical densities $\mu^n$ on $X_n\to Y_n$, $n\geq0$. If $n<m$ then $X_n\cong Y_n$ and $\mu^n$ is trivial. If $n>m$ then $\mu^n$ is the pullback vertical density of $\prod \mu^m$ on the cartesian square of  Lemma \ref{lemma:simple}:
$$\xymatrix@C=50pt{X_n \ar[d] \ar[r] & \prod X_m \ar[d] \\ 
Y_n \ar[r]_{y\mapsto \prod_\alpha \alpha^* y} & \prod Y_m}$$

Define $\mu_*:C^\infty(X_n)\to C^\infty(Y_n)$,
$\mu_*(\phi)(y) = \int_{f^{-1}(y)} \phi(x)\mu^n(x)$, the integration along the fiber. 
By Fubini's Theorem, $\mu_*d_i^*=d_i^*\mu_*$, so $\mu_*$ is a cochain map.
We claim that it is a quasi-inverse for $f^*$. Clearly  $\mu_*f^* = \id_{C^\infty(Y)}$. Let us show that
$f^*\mu_*$ and $\text{id}_{C^\infty(X)}$ are homotopic.


Write $PX$ for the {\bf path simplicial manifold}, whose $n$-simplices are given by $PX_n=X^{\Delta^n\times \Delta^1}$, and write 
$P(X/Y)=PX\times_{PY}Y$. Proposition \ref{prop:hom} shows that $PX_n$ is smooth and that $PX_n\to PY_n$ is a surjective submersion, so $PX$ and $P(X/Y)$ are well-defined and smooth. A simplex $\tilde x\in P(X/Y)_n$ is
given by a prism $\tilde x:\Delta^n\times I\to X$ such that $f\tilde x=y\pi:\Delta^n\times I\to Y$ for some $y\in Y_n$. Such a prism is given by its bottom face $x=\tilde x|_{\Delta^n\times 0}$ and the liftings of the non-degenerate $m$-simplices in $\Delta^n\times I\setminus\Delta^n\times 0$. In other words, we have a fibered product
$$\xymatrix{ P(X/Y)_n \ar[r] \ar[d]_s & \prod_{\substack{\alpha\in (\Delta^n\times I)_m\setminus\Delta^n\times 0\\\alpha\text{ non degenerate}}} X_m \ar[d] \\
X_n \ar[r] & \prod_{\substack{\alpha\in (\Delta^n\times I)_m\setminus\Delta^n\times 0\\\alpha\text{ non degenerate}}} Y_m}$$
We endow $s$ with the pullback vertical density induced by $\mu$, and define $\hat\mu_*:C(P(X/Y))\to C(X)$ by integration along the $s$-fibers. Clearly $\hat\mu_*s^*=\id:C(X)\to C(X)$. And by Fubini's Theorem,
\begin{align*}
\hat\mu_* t^*(\phi)(x) 
&=\int_{s^{-1}(x)} \phi(t(\tilde x)) d\mu(\tilde x)\\
&=\int_{f^{-1}(f(x))}\left(\int_{s^{-1}(x)\cap t^{-1}(x')} \phi(x') d\mu(\tilde x)\right)d\mu(x')\\
&=\int_{f^{-1}(f(x))} \phi(x') d\mu(x')=f^*\mu_*(\phi)(x) 
\end{align*}

The tautological simplicial homotopy $H:s\cong t:P(X/Y)\to X$ yields a chain homotopy $H^ *:s^*\cong t^*:C(X)\to C(P(X/Y))$. Explicitly, if $f\in C^n(X)$ and $\tilde x\in P(X/Y)_{n-1}$, then $H^*(f)(\tilde x)=\sum_j (-1)^j f(\tilde x P_j)$, where $P_j:\Delta^n\to\Delta^{n-1}\times I$, $P_j(i)=(i,0)$ if $i\leq j$, $P_j(i)=(i-1,1)$ if $i>j$. Finally $\hat\mu_* H^*:\hat\mu_*s=\id\cong\hat\mu_*t=f^*\mu_*:C(X)\to C(X)$, hence the result.
\end{proof}

\begin{corollary}\label{coro:cohomology-hypercover}
If $\phi:G\to H$ is a hypercover between Lie groupoids, and $V\to H$ is a higher vector bundle, then $\phi^*:C_{lin}(V)\to C_{lin}(\phi^*V)$ is a quasi-isomorphism.
\end{corollary}

\begin{proof}
The morphism $\phi^*V\to V$ is a hypercover by Lemma \ref{lemma:nice-family}, and by Theorem \ref{thm:higher-descent}, it yields a quasi-isomorphism
$C(V)=C_{lin}(V)\oplus K_V \to C(\phi^*V)=C_{lin}(\phi^*V)\oplus K_{\phi^*V}$, where the natural splitting is that of Lemma \ref{lemma:splitting}. The result follows.
\end{proof}


\begin{remark}
The Bott-Shulman complex $\Omega(X)=\Omega^q(X_p)$ of a simplicial manifold is the double complex computing its de Rham cohomology $H_{dR}(X)$. The previous proof can be adapted to show that a hypercover $f:X\to Y$ also yields an isomorphism $H_{dR}(Y)\cong H_{dR}(X)$, see \cite{getzler}. If $f:X\to Y$ is an $m$-simple hypercover, then one can define local sections of $X_m\to Y_m$ whose images contain the degenerate simplices. This reduces the problem to $m$-simple hypercovers with $X_m\to Y_m$ an open cover. Then $\mu$ is just a partition of 1, and the integration along the fiber of differential forms makes sense, it is just pasting the local pieces via $\mu$.
\end{remark}

\frenchspacing

\footnotesize{

}

\

\noindent
Matias del Hoyo; Universidade Federal Fluminense, Departamento de Geometria; Rua Prof. Marcos W. de Freitas Reis s/n, campus Gragoatá, 24210-201 Niterói, RJ, Brazil; mldelhoyo@id.uff.br

\

\noindent
Cristian Ortiz; Universidade de São Paulo, Instituto de Matemática e Estatística; Rua do Matão 1010, 05508-090, São Paulo, SP, Brazil; cortiz@ime.usp.br

\

\noindent
Fernando Studzinski; Universidade de São Paulo, Instituto de Matemática e Estatística; Rua do Matão 1010, 05508-090, São Paulo, SP, Brazil; ferstudz@usp.br

\end{document}